\documentclass[12pt,a4paper,leqno,twoside]{amsart}
\usepackage[a4paper,left=3cm,top=4cm,bottom=4cm,right=3cm]{geometry}
\usepackage[T1]{fontenc}
\usepackage{amssymb,bbm}
\usepackage{amsmath,amsthm,dsfont}
\usepackage{amsfonts,mathrsfs,wasysym}
\usepackage{latexsym}
\usepackage[cp1250]{inputenc}
\usepackage{graphicx}
\usepackage{indentfirst}
\usepackage{enumerate}
\usepackage{geometry}
\usepackage{enumitem}
\usepackage{xparse}
\usepackage{etoolbox}
\usepackage{mathptmx}
\usepackage{fancyhdr}
\usepackage[usenames,dvipsnames]{xcolor}
\usepackage{marvosym}
\usepackage{pifont}

\usepackage{hyperref}
\hypersetup{colorlinks,linkcolor={blue},citecolor={red}}  

\patchcmd{\abstract}{\scshape\abstractname}{\textbf{\abstractname}}{}{}

\makeatletter
\DeclareMathAlphabet{\mathcal}{OMS}{cmsy}{m}{n}

\DeclareSymbolFont{operators}{OT1}{ztmcm}{m}{n}
\DeclareSymbolFont{letters}{OML}{ztmcm}{m}{it}
\DeclareSymbolFont{symbols}{OMS}{ztmcm}{m}{n}
\DeclareSymbolFont{largesymbols}{OMX}{ztmcm}{m}{n}
\DeclareSymbolFont{bold}{OT1}{ptm}{bx}{n}
\DeclareSymbolFont{italic}{OT1}{ptm}{m}{it}
\@ifundefined{mathbf}{}{\DeclareMathAlphabet{\mathbf}{OT1}{ptm}{bx}{n}}
\@ifundefined{mathit}{}{\DeclareMathAlphabet{\mathit}{OT1}{ptm}{m}{it}}
\DeclareMathSymbol{\omicron}{0}{operators}{`\o}

\DeclareMathAlphabet{\mathpzc}{OT1}{pzc}{m}{it}
\DeclareSymbolFont{operators}{OT1}{txr}{m}{n}
\SetSymbolFont{operators}{bold}{OT1}{txr}{bx}{n}
\def\operator@font{\mathgroup\symoperators}
\DeclareSymbolFont{italic}{OT1}{txr}{m}{it}
\SetSymbolFont{italic}{bold}{OT1}{txr}{bx}{it}
\DeclareSymbolFontAlphabet{\mathrm}{operators}
\DeclareMathAlphabet{\mathbf}{OT1}{txr}{bx}{n}
\DeclareMathAlphabet{\mathit}{OT1}{txr}{m}{it}
\SetMathAlphabet{\mathit}{bold}{OT1}{txr}{bx}{it}
\DeclareSymbolFont{letters}{OML}{txmi}{m}{it}
\SetSymbolFont{letters}{bold}{OML}{txmi}{bx}{it}
\DeclareFontSubstitution{OML}{txmi}{m}{it}
\DeclareSymbolFont{lettersA}{U}{txmia}{m}{it}
\SetSymbolFont{lettersA}{bold}{U}{txmia}{bx}{it}
\DeclareFontSubstitution{U}{txmia}{m}{it}
\DeclareSymbolFontAlphabet{\mathfrak}{lettersA}
\DeclareSymbolFont{symbols}{OMS}{txsy}{m}{n}
\SetSymbolFont{symbols}{bold}{OMS}{txsy}{bx}{n}
\DeclareFontSubstitution{OMS}{txsy}{m}{n}
\makeatother

\makeatletter
\renewcommand\abstractname{\scshape\bfseries Abstract}

\renewenvironment{proof}[1][\proofname]{\par \pushQED{\qed} \normalfont
  \topsep6\p@\@plus6\p@ \trivlist \itemindent\z@
  \item[\hskip\labelsep\bfseries
    #1\@addpunct{.}]\ignorespaces
}{
  \popQED\endtrivlist\@endpefalse
}

\makeatother

\makeatletter
    \@addtoreset{equation}{section} 
    \renewcommand{\theequation}{{\thesection}.\@arabic\c@equation} 
\makeatother

\makeatletter

\def\section{\@ifstar\unnumberedsection\numberedsection}
\def\numberedsection{\@ifnextchar[
  \numberedsectionwithtwoarguments\numberedsectionwithoneargument}
\def\unnumberedsection{\@ifnextchar[
  \unnumberedsectionwithtwoarguments\unnumberedsectionwithoneargument}
\def\numberedsectionwithoneargument#1{\numberedsectionwithtwoarguments[#1]{#1}}
\def\unnumberedsectionwithoneargument#1{\unnumberedsectionwithtwoarguments[#1]{#1}}
\def\numberedsectionwithtwoarguments[#1]#2{%
  \ifhmode\par\fi
  \removelastskip
  \vskip 4ex\goodbreak
  \refstepcounter{section}%
  \noindent
  \begingroup
  \leavevmode\centering\scshape\bfseries
  \thesection.
  #2
  \par
  \endgroup
  \vskip 1ex\nobreak
  \addcontentsline{toc}{section}{%
    \protect\numberline{\thesection}%
    #1}%
  }
\def\unnumberedsectionwithtwoarguments[#1]#2{%
  \ifhmode\par\fi
  \removelastskip
  \vskip 2ex\goodbreak
  \noindent
  \begingroup
  \leavevmode\centering\scshape\bfseries
  \leavevmode\centering\scshape\bfseries
  #2
  \par
  \endgroup
  \vskip 1ex\nobreak
  \addcontentsline{toc}{section}{%
    #1}%
}
\makeatother

\makeatletter
\def\@seccntformat#1{\csname mythe#1\endcsname}

\let\latex@subsection\subsection
\def\subsection{\@ifstar{\refstepcounter{subsection}\latex@subsection*}{\latex@subsection}}
\def\@makechapterhead#1{%
  \vspace*{40\p@}%
  {\parindent \z@ \raggedright \normalfont
    \interlinepenalty\@M
    \Huge \bfseries #1\par \nobreak
    \vskip 40\p@
  }}
\let\latex@l@chapter\l@chapter
\def\l@chapter#1#2{\begingroup\let\numberline\@gobble\latex@l@chapter{#1}{#2}\endgroup}
\makeatother

\theoremstyle{plain}
\newtheorem{Th}{Theorem}[section]
\newtheorem{Prop}[Th]{Proposition}
\newtheorem{Lem}[Th]{Lemma}
\newtheorem{Cor}[Th]{Corollary}
\theoremstyle{definition}
\newtheorem{Rem}[Th]{Remark}
\newtheorem{Ex}[Th]{Example}

\newtheorem{Def}[Th]{Definition}

\def\bf{\textbf}
\def\it{\textit}
\def\tn{\textnormal}

\def\leq{\leqslant}
\def\geq{\geqslant}
\def\R{{\mathds R}}
\def\N{{\mathds N}}

\def\B{\textit{I\!B}}

\def\D{{\mathrm{dom}}\,}
\def\E{{\mathrm{epi}}\,}

\def\G{{\mathrm{gph}}\,}

\def\dist{\mathrm{dist}}
\def\e{{\scriptstyle\Theta}}

\begin{document}
\vspace*{0cm}
\title{Reduction of lower semicontinuous solutions of Hamilton-Jacobi-Bellman equations}

\author{\vspace*{-0.2cm}{Arkadiusz Misztela \textdagger}\vspace*{-0.2cm}}
\thanks{\textdagger\, Institute of Mathematics, University of Szczecin, Wielkopolska 15, 70-451 Szczecin, Poland; e-mail: arkadiusz.misztela@usz.edu.pl}

\begin{abstract} 
This article is devoted to the study of lower semicontinuous solutions of Hamilton-Jacobi  equations with convex Hamiltonians in a gradient variable. Such\linebreak Hamiltonians appear in the optimal control theory. We present a necessary and sufficient condition for a reduction of a Hamiltonian satisfying optimality conditions to the case when the Hamiltonian is positively homogeneous and also satisfies optimality conditions. It allows us to reduce some uniqueness problems of lower semicontinuous solutions to Barron-Jensen and Frankowska theorems. For  Hamiltonians, which cannot be reduced in that way, we prove the new existence and uniqueness theorems.\\
\vspace{0mm}

\hspace{-1cm}
\hspace*{-0.55cm} set-valued analysis, nonsmooth analysis, convex analysis.

\vspace{3mm}\hspace{-1cm}
\end{abstract}

\maketitle

\pagestyle{myheadings}  \markboth{\small{\scshape Arkadiusz Misztela}
}{\small{\scshape Hamilton-Jacobi-Bellman Equations}}

\thispagestyle{empty}

\vspace{-0.8cm}


\section{Introduction}

\noindent The Cauchy problem for the  Hamilton-Jacobi equation 
\begin{equation}\label{rowhj}
\begin{array}{rll}
-U_{t}+ H(t,x,-U_{x})=0 &\!\!\textnormal{in}\!\! & (0,T)\times\R^{\scriptscriptstyle N}, \\[0mm]
U(T,x)=g(x) & \!\!\textnormal{in}\!\! & \;\R^{\scriptscriptstyle N},
\end{array}
\end{equation}
with a  convex Hamiltonian $H$  in the gradient variable can be studied with connection to 
 a calculus of variations problem. Let  $H^{\ast}$ be the Legendre-Fenchel conjugate of $H$ in its gradient variable (in our case $H^{\ast}$ is an extended real-valued function):
\begin{equation*}
H^{\ast}(t,x,v)= \sup_{p\in\R^{\scriptscriptstyle N}}\,\{\,\langle v,p\rangle-H(t,x,p)\,\}.
\end{equation*}
Here $\langle v,p\rangle$ denotes the inner product of $v$ and $p$. We will use the notation $\D H^{\ast}(t,x,\cdot)$ for the \emph{effective domain} of  $H^{\ast}(t,x,\cdot)$, which is a set of all $v$ such that $H^{\ast}(t,x,v)\not=\pm\infty$.
The value function of a  calculus of variations problem is defined by the formula
\begin{equation}\label{fwwp}
V(t_0,x_0)= \inf_{\begin{array}{c}
\scriptstyle x(\cdot)\,\in\,\mathcal{A}\left([t_0,T],\R^{\scriptscriptstyle N}\right)\\[-1mm]
\scriptstyle x(t_0)=x_0
\end{array}}\,\Big\{\,g(x(T))+\int_{t_0}^TH^{\ast}(t,x(t),\dot{x}(t))\,dt\,\Big\},
\end{equation}
where $\mathcal{A}\!\left([t_0,T],\R^{\scriptscriptstyle N}\right)$ denotes the space of all absolutely continuous functions from $[t_0,T]$ into $\R^{\scriptscriptstyle N}$.
If the value function is real-valued and differentiable, it is well-known that it satisfies \eqref{rowhj} in the classical sense. However, in many situations the value function is extended real-valued and merely lower semicontinuous. Then the solution of \eqref{rowhj} must be defined in a nonsmooth sense in such a way that under quite general assumptions on $H$ and $g$, $V$ is the unique solution of \eqref{rowhj}.

\pagebreak
 We consider the following optimality conditions:
\begin{enumerate}[leftmargin=9.7mm]
\item[\tn{\bf{(H1)}}] $H:[0,T]\times\R^{\scriptscriptstyle N}\times\R^{\scriptscriptstyle N}\rightarrow\R$ is continuous with
respect to all variables;
\item[\tn{\bf{(H2)}}] $H(t,x,p)$ is convex with respect to $p$ for every $t\in[0,T]$ and 
$x\in\R^{\scriptscriptstyle N}$;
\item[\tn{\bf{(H3)}}] For any $R\geq 0$ there exists $C_R\geq 0$ such that for all $t\in[0,T]$, $x\in\B_R$ \\
and $p,q\in\R^{\scriptscriptstyle N}$ one has $|H(t,x,p)-H(t,x,q)|\leq C_R|p-q|$;
\item[\tn{\bf{(H4)}}] There exists an integrable function $c:[0,T]\to[0,+\infty)$ such that for almost all\\ 
$t\in[0,T]$ and all $x,p,q\in\R^{\scriptscriptstyle N}$ one has $|H(t,x,p)-H(t,x,q)|\leq c(t)(1+|x|)|p-q|$;
\item[\tn{\bf{(H5)}}] For any $R\geq 0$ there exists an integrable function $k_R:[0,T]\to[0,+\infty)$ such that\\ $|H(t,x,p)-H(t,y,p)|\leq k_R(t)(1+|p|)|x-y|$ for all $x,y\!\in\!\B_R$, $p\!\in\!\R^{\scriptscriptstyle N}$ and a.e. $\!t\in[0,T]$.
\end{enumerate}
where  $\B_R$ denotes a closed ball in  $\R^{\scriptscriptstyle N}$ of center $0$ and radius $R\geq 0$ and $|\cdot|$  denotes the Euclidean norm on  $\R^{\scriptscriptstyle N}$. For a nonempty subset $W$ of $\R^{\scriptscriptstyle N}$ we define $\|W\|:=\sup_{\xi\in W}|\xi|$.

\vspace{2mm}
Barron-Jensen \cite{B-J} and Frankowska \cite{HF} studied extended viscosity solutions to semicontinuous functions for Hamiltonians that are convex with respect to the last variable.   Frankowska~\cite{HF}  called these solutions   \it{lower semicontinuous solutions}.

\begin{Def}\label{lsc-solutions}  A function $U:[0,T]\times\R^{\scriptscriptstyle N}\rightarrow\R\cup\{+\infty\}$ is a \it{lower semicontinuous solution} of  \eqref{rowhj} if $U$ is a lower semicontinuous function, $U(T,x)=g(x)$ for all $x\in\R^{\scriptscriptstyle N}$, and for any $(t,x)\in \D U$, for all $(p_{t},p_{x})\in\partial U(t,x)$,  one has
\begin{align}
& -p_{t}+H(t,x,-p_{x})\geq 0\;\;\;\;\;\tn{if}\;\;\;\;\;0\leq t< T,\label{lsc-solutions-1}\\
& -p_{t}+H(t,x,-p_{x})\leq 0\;\;\;\;\;\tn{if}\;\;\;\;\;0<t\leq T,\label{lsc-solutions-2}
\end{align}
where $\partial U(t,x)$ denotes subdifferential of $U$ at $(t,x)$.
\end{Def}

Frankowska \cite{HF} proved that the value function $V$ is the unique lower semicontinuous solution of   \eqref{rowhj} if the Hamiltonian $H$ satisfies (H1)-(H5) and it is positively homogeneous in $p$, i.e. $\forall _{r\geq 0}\;H(t,x,rp)=rH(t,x,p)$. Whereas  $g$ is a lower semicontinuous extended real-valued function  which  does not take on the value $-\infty$. Earlier, Barron-Jensen \cite{B-J,B-J-1991} using quite different methods obtained similar results to those of Frankowska  assuming slightly stronger conditions on the Hamiltonian $H$. The paper by Barles \cite{GB} provides some extensions and an informal discussion of Barron and Jensen's ideas. 

\vspace{2mm}
For lower semicontinuous solutions the uniqueness results of \cite{B-J,HF} use   positively homogeneous Hamiltonian in $p$. This is not quite as restrictive as it seems at first. Because  $H:[0,T]\times\R^{\scriptscriptstyle N}\times\R^{\scriptscriptstyle N}\rightarrow\R$ can be replaced by $\bar{H}:[0,T]\times\R^{\scriptscriptstyle N+1}\times\R^{\scriptscriptstyle N+1}\rightarrow\R$ such that
\begin{equation}\label{phc}\begin{split}
&\bar{H}(t,x,r,p,q)\;\tn{is positively homogeneous in}\;(p,q)\; \tn{for all}\;t\!\in\![0,\!T],x\!\in\!\R^{\scriptscriptstyle N}\!\!,r\!\in\!\R,\\[-1.5mm]
&\bar{H}(t,x,r,p,-1)=H(t,x,p)\;\tn{for all}\;t\!\in\![0,T],x\!\in\!\R^{\scriptscriptstyle N}\!\!,r\!\in\!\R,p\!\in\!\R^{\scriptscriptstyle N}.
\end{split}
\end{equation}
Indeed, suppose that there are  two lower semicontinuous solutions $U_1$ and $U_2$ of \eqref{rowhj} with $H$ and $g$. Then, $\bar{U}_1(t,x,r)=U_1(t,x)+r$ and $\bar{U}_2(t,x,r)=U_2(t,x)+r$ both satisfy \eqref{rowhj} with $\bar{H}$ and $\bar{g}(x,r)=g(x)+r$. In view of the uniqueness result of \cite{HF} we get $\bar{U}_1=\bar{U}_2$. Hence it follows that $U_1=U_2$. It is possible, provided that $\bar{H}$ satisfies not only \eqref{phc}, but also $(\bar{\tn{H}}1)$-$(\bar{\tn{H}}5)$, where $(\bar{\tn{H}}1)$-$(\bar{\tn{H}}5)$ denotes conditions (H1)-(H5) for $\bar{H}$ with doubled arguments $(x,r)$ and $(p,q)$. Therefore, it should be stated precisely what conditions on $H$ imply the existence of $\bar{H}$ which satisfies conditions \eqref{phc} and $(\bar{\tn{H}}1)$-$(\bar{\tn{H}}5)$. We define such conditions on the Hamiltonian $H$ in the following theorem. 

\begin{Th}\label{thmrd}
Let $H$ be given. Then the following conditions are equivalent:
\begin{enumerate}[leftmargin=7.5mm]
\item[\tn{\bf{(A)}}] $H$ satisfies \tn{(H1)-(H5)} and there is a continuous function $\lambda:\![0,T]\times\R^{\scriptscriptstyle N}\to\R^+$ such that $\|\D H^{\ast}(t,x,\cdot)\|\leq\lambda(t,x)$ and  $\|H^{\ast}(t,x,\D H^{\ast}(t,x,\cdot))\|\leq\lambda(t,x)$ for all $t\!\in\![0,\!T]$, $x\!\in\!\R^{\scriptscriptstyle N}$\tn{;}\linebreak $\forall R\!\geq\!0\;\exists\,\zeta_R(\cdot)\!\in\! L^{\scriptscriptstyle 1}([0,\!T],\R^+)$ such that $\lambda(t,\cdot)$ is $\zeta_R(t)$-Lipschitz on $\B_R$ for a.e. $t\in[0,\!T]$\tn{;}\linebreak
$\exists\,\vartheta(\cdot)\in L^{\scriptscriptstyle 1}([0,\!T],\R^+)$ such that $\lambda(t,x)\leq\vartheta(t)(1+|x|)$ for all $x\in\R^{\scriptscriptstyle N}$ and a.e. $t\in[0,T]$.
\item[\tn{\bf{(B)}}] There exists $\bar{H}$ satisfying  \eqref{phc} and $(\bar{\tn{H}}1)$-$(\bar{\tn{H}}5)$.
\end{enumerate}
\end{Th}

From the above theorem it follows that boundedness of the sets  $H^{\ast}(t,x,\D H^{\ast}(t,x,\cdot))$ and $\D H^{\ast}(t,x,\cdot)$ by an appropriately regular function $\lambda$ is a condition not only sufficient\linebreak but also necessary for the reduction described above. We prove Theorem \ref{thmrd} using\linebreak recently derived results \cite{AM,AM1} concerning representations of Hamilton-Jacobi equations in the optimal control theory; see Section \ref{section-3}. In papers
 \cite{DM-F-V,G,AM2,P-Q,FR-S} one investigates\linebreak  existence and uniqueness of lower semicontinuous solutions of \eqref{rowhj}, assuming a weak growth and a weak Lipschitz continuity for the Hamiltonian $H$. These nonrestrictive\linebreak conditions imply unboundedness of the sets $\D H^{\ast}(t,x,\cdot)$ and $H^{\ast}(t,x,\D H^{\ast}(t,x,\cdot))$. Thus, such Hamiltonian
does not satisfy the condition (A).
The conditions (H3) and (H4) imposed on the Hamiltonian $H$ imply some type of boundedness of the set $\D H^{\ast}(t,x,\cdot)$. However, there exists a large class of Hamiltonians $H$, that satisfy
 (H1)-(H5), but do not have bounded sets  $H^{\ast}(t,x,\D H^{\ast}(t,x,\cdot))$; see Section \ref{section-2}. This kind of Hamiltonians\linebreak  derive from the optimal control problems with unbounded control set; see Example \ref{ex-5}. We also give an example of Hamiltonian $H$ with the bounded sets $\D H^{\ast}(t,x,\cdot)$ and $H^{\ast}(t,x,\D H^{\ast}(t,x,\cdot))$ such that an appropriately regular function $\lambda$  bounding them does not exist; see Example \ref{ex-2}. 
Thus, we see that there exists a large class of Hamiltonians, that satisfy conditions (H1)-(H5), but do not fulfill the condition (A). For such kind of Hamiltonians we prove the following theorem.

\begin{Th}\label{eauhje}
 Let $g$ be a lower semicontinuous extended real-valued function  which  does not take on the value $-\infty$. Assume that $H$ satisfies \tn{(H1)-(H5)}. If $V$ is the value function associated with $H^{\ast}$ and $g$, then  $V$ is a lower semicontinuous solution of \eqref{rowhj}. Moreover, if $U$ is a lower semicontinuous solution of \eqref{rowhj}, then $U=V$ on $[0,T]\times\R^{\scriptscriptstyle N}$.
\end{Th}

We prove Theorem \ref{eauhje} using methods of  the viability theory similarly to Frankowska \cite{HF}; see Sections \ref{section-4}-\ref{section-5}. The difference is that we consider set-valued maps with unbounded values. Whereas in paper \cite{HF} set-valued maps with compact values are considered. These differences cause new difficulties, but we are able to deal with them. Thus, we obtain\linebreak result that do not need the positively homogeneous assumption on the Hamiltonian in $p$.\linebreak In papers \cite{B-CD,B-B,C-S-2004,C-L-87,HI0,HI,AM3, AP} one can find similar results to Theorem \ref{eauhje}. However,\linebreak these results usually require that the solution is continuous and bounded and that the Hamiltonian satisfies some kind of uniform continuity conditions with constant functions $c(\cdot)$ and $k_R(\cdot)$. Galbraith \cite{G} also obtained similar results to Theorem~\ref{eauhje}. However, his methods need a strong Lipschitz-type assumption on the Hamiltonian with respect to the time variable. In the literature related to such results one usually assumes that the Hamiltonian is continuous in the time variable \cite{HF} or only measurable \cite{F-P-Rz}. Recently, there has been a paper  \cite{B-V} assuming that the Hamiltonian is discontinuous with respect to the time variable in the following sense: it has everywhere left and right limits and  is continuous on a set of the full measure. Finally, it is worth considering whether the positively homogeneous assumption on the Hamiltonian in $p$ can be removed in papers \cite{B-V,F-M-2013, F-P-Rz} using the methods of this paper. At the moment we do not know if it is possible.

\vspace{2mm}
The outline of the paper is as follows. Section \ref{section-2} contains a preliminary material and examples. In Section \ref{section-3} we prove Theorem \ref{thmrd}. Section~\ref{section-4} contains new viability and invariance theorems along with their proofs. In Section \ref{section-5} using viability and invariance theorems from Section \ref{section-4}  we prove Theorem \ref{eauhje}.

\section{Preliminary Material and Examples}\label{section-2}

\noindent Let $\overline{\R}=\R\cup\{\pm\infty\}$ and $\varphi:\R^{\scriptscriptstyle M}\to\overline{\R}$ be a function. The sets: $\D\varphi=\{\,z\in\R^{\scriptscriptstyle M}\mid\varphi(z)\not=\pm\infty\,\}$, $\G\varphi=\{\,(z,r)\in\R^{\scriptscriptstyle M}\times\R\mid\varphi(z)=r\,\}$ and $\E\varphi=\{\,(z,r)\in\R^{\scriptscriptstyle M}\times\R\mid\varphi(z)\leq r\,\}$ are called the \emph{effective domain}, the \it{graph} and the \it{epigraph} of $\varphi$, respectively. We say that $\varphi$ is \it{proper} if it never takes  the value $-\infty$ and it is not identically equal to $+\infty$.  Using properties of the Legendre-Fenchel conjugate from~\cite{R-W} we can prove the following proposition.

\begin{Prop}\label{prop2-fmw} Assume that $H$ satisfies \tn{(H1)-(H2)}. If $L(t,x,\cdot\,)=H^{\ast}(t,x,\cdot\,)$, then
\begin{enumerate}
\item[\tn{\bf{(L1)}}] $L:[0,T]\times\R^{\scriptscriptstyle N}\times\R^{\scriptscriptstyle N}\to\overline{\R}$ is lower semicontinuous with respect to all variables\tn{;}
\item[\tn{\bf{(L2)}}] $L(t,x,v)$ is convex and proper with respect to $v$ for every $t\in[0,T]$ and 
$x\in\R^{\scriptscriptstyle N}$\tn{;}
\item[\tn{\bf{(L3)}}] $\forall\,(t,x,v)\in[0,T]\times\R^{\scriptscriptstyle N}\times\R^{\scriptscriptstyle N}\;\forall\,(t_n,x_n)\rightarrow (t,x)\;\exists\,v_n\rightarrow v\,:\,L(t_n,x_n,v_n)\rightarrow L(t,x,v)$.
\item[]\hspace{-1.3cm}Additionally, if $H$ satisfies \tn{(H3)}, then the following property holds
\item[\tn{\bf{(L4)}}] $\forall\,R\geq 0\;\exists\,C_R\geq 0\;\forall\,(t,x,v)\in[0,T]\times\B_R\times\R^{\scriptscriptstyle N} \,:\, |v|>C_R\;\Rightarrow\; L(t,x,v) =+\infty$.
\item[]\hspace{-1.3cm}Additionally, if $H$ satisfies \tn{(H4)}, then there exists a measure zero set $\mathcal{N}$ such that
\item[\tn{\bf{(L5)}}] $\forall\,(t,x,v)\in[0,T]\setminus\mathcal{N}\times\R^{\scriptscriptstyle N}\times\R^{\scriptscriptstyle N} \;:\; |v|>c(t)(1+|x|)\;\Rightarrow\; L(t,x,v) =+\infty$.
\end{enumerate}
\end{Prop}

Actually, we can prove that (H1)-(H4) are equivalent to  (L1)-(L5).
  The set $\G S:= \{\,(z,w)\in\R^{\scriptscriptstyle M}\times\R^{\scriptscriptstyle N}\mid w\in S\!(z)\,\}$ is called a \it{graph} of the set-valued map $S:\R^{\scriptscriptstyle M}\multimap\R^{\scriptscriptstyle N}$. A~set-valued map $S:\R^{\scriptscriptstyle M}\multimap\R^{\scriptscriptstyle N}$ is \it{lower semicontinuous} in  Kuratowski's sense if for each open set $O\subset\R^{\scriptscriptstyle N}$ the inverse image  $S^{-1}(O):= \{\,z\in\R^{\scriptscriptstyle M}\mid S\!(z)\cap O\not=\emptyset\,\}$ is open in $\R^{\scriptscriptstyle M}$. It is equivalent to  $\forall\,(z,w)\in\G S\;\forall\,z_n\to z\;\exists\,w_n\to w\,:\,w_n\in S\!(z_n)$ for large $n\in\N$.     

\vspace{2mm}
Let us define the set-valued map $Q:[0,T]\times\R^{\scriptscriptstyle N}\multimap\R^{\scriptscriptstyle N}\times\R$ by the formula
\begin{equation*}
Q(t,x):= \{\,(v,\eta)\in\R^{\scriptscriptstyle N}\times\R\,\mid\, (v,-\eta)\in\E L(t,x,\cdot)\,\}.
\end{equation*}

From results in \cite[Chap. 5]{R-W} we deduce the following corollary. 

\begin{Cor}\label{wrow-wm}
If $L$ satisfies \tn{(L1)-(L3)}, then
\begin{enumerate}
\item[$\pmb{(\mathcal{Q}1)}$] the set-valued map $(t,x)\to Q(t,x)$ has nonempty, closed, convex values;
\item[$\pmb{(\mathcal{Q}2)}$] the set-valued map $(t,x)\to Q(t,x)$ is lower semicontinuous;
\item[$\pmb{(\mathcal{Q}3)}$] the set-valued map $(t,x)\to Q(t,x)$ has a closed graph.
\item[]\hspace{-1.3cm}Additionally, if $L$ satisfies \tn{(L4)}, then the following inequality holds
\item[$\pmb{(\mathcal{Q}4)}$] $\|\D L(t,x,\cdot)\|\leq C_R$\, for every $(t,x)\in[0,T]\times\B_R$ and  $R\geq 0$.
\item[]\hspace{-1.3cm}Additionally, if $L$ satisfies \tn{(L5)}, then the following inequality holds
\item[$\pmb{(\mathcal{Q}5)}$] $\|\D L(t,x,\cdot)\|\leq c(t)(1+|x|)$ for  almost all $t\in[0,T]$ and every $x\in\R^{\scriptscriptstyle N}$.
\end{enumerate}
\end{Cor}

We present Hausdorff continuity of a set-valued map in Lagrangian and Hamiltonian terms. For nonempty subsets $V$, $W$ of $\R^{\scriptscriptstyle M}$ and a number $r\in\R$ we define $rW:=\{r\,w\mid w\in W\}$ and $V\!+\!W:=\{v\!+\!w\mid v\in V,\, w\in W\}$. Set $\B(z,r):=\{z\}+r\B$, where $z\!\in\!\R^{\scriptscriptstyle M}$, $\B:=\B_1$, $r\geq 0$.

\begin{Th}[\tn{\cite[Thm. 2.3]{AM}}]\label{tw2_rlhmh}
Assume that $H$ satisfies \tn{(H1)-(H4)} or equivalently $L$ satisfies \tn{(L1)-(L5)}. Let  $L(t,x,\cdot\,)=H^{\ast}(t,x,\cdot\,)$ and  $H(t,x,\cdot\,)=L^{\ast}(t,x,\cdot\,)$. Then there are the equivalences $\tn{(H5)}\Leftrightarrow\tn{(L6)}\Leftrightarrow (\mathcal{Q}6)$\tn{:}

\tn{\bf{(L6)}} For any $R\geq 0$ there exists an integrable  map $k_R:[0,T]\to[0,+\infty)$  such that for almost all $t\in[0,T]$  and every $x,y\in \B_R$, $v\in\D L(t,x,\cdot)$ there exists $w\in\D L(t,y,\cdot)$ satisfying inequalities $|w-v|\leq k_R(t)|y-x|$ and $L(t,y,w)\leq L(t,x,v)+k_R(t)|y-x|$.

$\pmb{(\mathcal{Q}6)}$ For any $R\geq 0$ there exists an integrable map $k_R:[0,T]\to[0,+\infty)$ such that $Q(t,x)\,\subset\, Q(t,y)+k_R(t)\,|x-y|\,(\B\times[-1,1])$ for almost all $t\in[0,T]$ and every $x,y\in\B_R$.
\end{Th}

\begin{Lem}[\tn{Cesari \cite[Sects. 8.5 and 10.5]{LC}}]\label{llcq}
Assume that $L$ satisfies \tn{(L1)-(L4)}. Then the set-valued map $Q$ for all $t\in[0,T]$ and $x\in\R^{\scriptscriptstyle N}$ has the following
property 
\begin{align*}
& Q(t,x)=\bigcap_{\varepsilon\,>\,0}\,\mathrm{cl}\,\mathrm{conv}\,Q(t,x;\varepsilon), \; \tn{where}\\
& Q(t,x;\varepsilon):=\bigcup_{|t-s|\,<\,\varepsilon,\;|x-y|\,<\,\varepsilon}Q(s,y).
\end{align*}
\end{Lem}

\subsection{Nonsmooth Analysis}\label{nap} The  distance from the point $y\in\R^{\scriptscriptstyle M}$ to the nonempty subset $E$ of $\R^{\scriptscriptstyle M}$  is defined by $\dist(y,E):=\inf_{w\in E}|\,y-w\,|$. For a function $\varphi:\R^{\scriptscriptstyle M}\to\overline{\R}$  and a point $z\in\D\varphi$, the \it{subderivative function} $d\varphi(z):\R^{\scriptscriptstyle M}\to\overline{\R}$  is defined by
\begin{equation*}
d\varphi(z)(v):=\liminf_{\tau\to 0+,\,y\to v}\frac{\varphi(z+\tau y)-\varphi(z)}{\tau}.
\end{equation*}
The \it{subdifferential} of the function $\varphi:\R^{\scriptscriptstyle M}\to\overline{\R}$ at the point $z\in\D\varphi$ is defined by
\begin{equation*}
\partial\varphi(z):=\big\{\,p\in\R^{\scriptscriptstyle M}\,\mid\,\langle v,p\rangle\leq d\varphi(z)(v)\;\tn{for all}\;v\in\R^{\scriptscriptstyle M}\,\big\}.
\end{equation*}
The \it{tangent cane} to the subset $E$ of $\R^{\scriptscriptstyle M}$ at the point $w\in E$ is defined by
\begin{equation*}
T_{\!E}(w):=\Big\{\,\zeta\in\R^{\scriptscriptstyle M}\,\mid\,\liminf_{\tau\to 0+}\,\frac{\dist(w+\tau \zeta,E)}{\tau}=0 \,\Big\}.
\end{equation*}
We define the \it{normal cone} to the subset $E$ of $\R^{\scriptscriptstyle M}$ at the point $w\!\in\! E$ by polarity with $T_{\!E}(w)$:
\begin{equation*}
N_E(w):=\big\{\,\xi\in\R^{\scriptscriptstyle M}\,\mid\,\langle \zeta,\xi\rangle\leq 0\;\tn{for all}\;\zeta\in T_{\!E}(w)\,\big\}.
\end{equation*}
It follows from \cite[Prop. 6.5 and Ex. 6.16]{R-W} that $y-w\in N_E(w)$ whenever $w\in E$ and $|y-w|=\dist(y,E)$. Moreover, it follows from \cite[Ex. 8.4 and Thm. 8.9]{R-W} that $p\in\partial\varphi(z)$ if and only if $(p,-1)\in N_{\E\varphi}(z,\varphi(z))$. By the definition of the normal cone for all $(p,q)\in N_{\E\varphi}(z,\varphi(z))$ we have $q\leq 0$. Moreover, $N_{\E\varphi}(z,r)\subset N_{\E\varphi}(z,\varphi(z))$ for all  $(z,r)\in\E \varphi$.

\begin{Lem}[Rockafellar]\label{lemrlsc}
Assume that  $\varphi:\R^{\scriptscriptstyle M}\to\overline{\R}$ is a proper and lower semicontinuous  function. Let $z\in\D\varphi$ and $(p,0)\in N_{\E\varphi}(z,\varphi(z))$. Then there exist  $z_k\to z$, $p_k\to p$, $q_k\to 0$  with $\varphi(z_k)\to\varphi(z)$ satisfying $q_k<0$ and $(p_k,q_k)\in N_{\E\varphi}(z_k,\varphi(z_k))$ for all $k\in\N$.
\end{Lem}

\subsection{Examples} Now we present examples of Hamiltonians which satisfy (H1)-(H5). These examples have nonregular  Lagrangians, so they  do not satisfy  the condition (A).\linebreak It means that these Hamiltonians are not liable to  the reduction described in the introduction. However, they satisfy the assumptions of Theorem \ref{eauhje}.

\begin{Ex}\label{ex-1}
Let us define the Hamiltonian $H:\R\times\R\to\R$ by the formula
\begin{equation*}
H(x,p):=\left\{
\begin{array}{ccl}
(\sqrt{|xp|}-1)^2 & \tn{if} & |xp|\,> 1, \\[1mm]
0 & \tn{if} & |xp|\leq 1.
\end{array}
\right.
\end{equation*}
This Hamiltonian satisfies conditions  (H1)-(H5). Moreover, $L(x,\cdot\,)=H^{\ast}(x,\cdot\,)$ has the form
\begin{equation*}
L(x,v)=\left\{
\begin{array}{ccl}
+\infty & \tn{if} & v\not\in(-|x|,|x|\,),\;x\not=0,\\[1mm]
\frac{\displaystyle |v|}{\displaystyle |x|-|v|} & \tn{if} & v\in(-|x|,|x|\,),\;x\not=0, \\[2mm]
0 & \tn{if} & v=0,\; x=0,\\[0mm]
+\infty & \tn{if} & v\not=0,\;x=0.
\end{array}
\right.
\end{equation*}
The function $v\rightarrow L(x,v)$ is not bounded on  $\D L(x,\cdot)=(-|x|,|x|\,)$ for every $x\in\R\setminus\{0\}$. Therefore,  a real-valued  function $\lambda$ such that $L(x,v)\leq\lambda(x)$ for all $v\in\D L(x,\cdot)$ and $x\in\R$ does not exist. So this Hamiltonian  does not satisfy the condition (A).
\end{Ex}

\begin{Ex}\label{ex-2}
Let us define the Hamiltonian $H:[0,1]\times\R\to\R$ by the formula
\begin{equation*}
H(t,p):=\left\{
\begin{array}{ccl}
\max\left\{\,|p|-\frac{1}{\sqrt{t}},\,0\,\right\} & \tn{if} & p\in\R,\;t\not=0, \\[1mm]
0 & \tn{if} & p\in\R,\;t=0.
\end{array}
\right.
\end{equation*}
This Hamiltonian satisfies conditions  (H1)-(H5). Moreover, $L(t,\cdot\,)=H^{\ast}(t,\cdot\,)$ has the form
\begin{equation*}
L(t,v)=\left\{
\begin{array}{ccl}
+\infty & \tn{if} & v\not\in[-1,1],\;t\not=0,\\[1mm]
\frac{\displaystyle |v|}{\displaystyle \sqrt{t}} & \tn{if} & v\in[-1,1],\;t\not=0, \\[3mm]
0 & \tn{if} & v=0,\;t=0,\\[0mm]
+\infty & \tn{if} & v\not=0,\;t=0.
\end{array}
\right.
\end{equation*}
Let $\lambda(t)=1/\sqrt{t}$ for $t\in(0,1]$ and $\lambda(0)=0$. Observe that $\|\D L(t,\cdot)\|\leq 1$ for all $t\in[0,1]$ and $\|L(t,\D L(t,\cdot))\|=\lambda(t)$ for all $t\in[0,1]$. Therefore, both the sets $\D L(t,\cdot)$ and $L(t,\D L(t,\cdot))$ are bounded  for all $t\in[0,1]$. However, the function $\lambda$ on the set $[0,1]$ is unbounded.  So this Hamiltonian  does not satisfy the condition (A).
\end{Ex}

\begin{Ex}\label{ex-3}
Let us define the function $\alpha:[0,1]\times\R\to[0,+\infty)$ by the formula
\begin{equation*}
\alpha(t,x):=\left\{
\begin{array}{ccl}
\max\left\{\,\frac{|x|}{\sqrt{t}}-\frac{1}{t},\,0\,\right\} & \tn{if} & x\in\R,\;t\not=0, \\[1mm]
0 & \tn{if} & x\in\R,\;t=0.
\end{array}
\right.
\end{equation*}
The function $\alpha$ is locally Lipschitz continuous. Let $c(t)=1/\sqrt{t}$ for $t\in(0,1]$ and $c(0)=0$. Then $c$ is an integrable function. Moreover, $\alpha(t,x)\leq c(t)(1+|x|)$ for all $t\in[0,1]$, $x\in\R$. However, there is no such constant $c$ that  $\alpha(t,x)\leq c(1+|x|)$ for all $x\in\R$ and a.e. $t\in[0,1]$. 

Let us define the function $\beta:[0,1]\times\R\to[0,+\infty)$ by the formula
\begin{equation*}
\beta(t,x):=\left\{
\begin{array}{ccl}
\big(\,\sqrt{t}+|x|\,\big)\,\Big|\sin\!\left(\displaystyle\frac{1}{\sqrt{t}+|x|}\right)\!\!\Big| & \tn{if} & (t,x)\not=(0,0), \\[1mm]
0 & \tn{if} & (t,x)=(0,0).
\end{array}
\right.
\end{equation*}
The function $\beta$ is continuous. Let $k(t)=2/\sqrt{t}$ for $t\in(0,1]$ and $k(0)=0$. Then $k$ is an integrable function. Moreover, $|\,\beta(t,x)-\beta(t,y)\,|\leq k(t)\,|x-y|$ for all $t\in(0,1]$, $x,y\in\R$. Additionally, $\beta(t,x)\leq 1+|x|$ for all $t\in[0,1]$, $x\in\R$. However, there is no such constant $k_R$ that  $|\,\beta(t,x)-\beta(t,y)\,|\leq k_R\,|x-y|$ for all $x,y\in\B_R$ and a.e. $t\in[0,1]$. 

Let us define the function $\gamma:[0,1]\times\R\to[0,+\infty)$ by the formula
\begin{equation*}
\gamma(t,x)\,:=\,\alpha(t,x)\;+\;\beta(t,x).
\end{equation*}
The function $\gamma$ is continuous. Moreover, for any $R\geq 0$ there exists an integrable function $k_R:[0,1]\to\R^+$ such that $\gamma(t,\cdot)$ is $k_R(t)$-Lipschitz on $\B_R$ for a.e. $t\in[0,1]$.
Additionally, there exists an integrable function $c:[0,1]\to\R^+$ such that $\gamma(t,x)\leq c(t)(1+|x|)$ for  all $x\in\R$ and a.e. $t\in[0,1]$. However, the functions $k_R(\cdot)$ and $c(\cdot)$ cannot be bounded.
\end{Ex}

\begin{Ex}\label{ex-4}
Let us define the Hamiltonian $H:[0,1]\times\R\times\R\to\R$ by the formula
\begin{equation*}
H(t,x,p):=\left\{
\begin{array}{ccl}
(\sqrt{\gamma(t,x)\,|p|}-1)^2 & \tn{if} & \gamma(t,x)\,|p|\,> 1, \\[1mm]
0 & \tn{if} & \gamma(t,x)\,|p|\leq 1,
\end{array}
\right.
\end{equation*}
where $\gamma$ is defined as in Example \ref{ex-3}.
This Hamiltonian satisfies conditions  (H1)-(H5) with the unbounded functions $k_R(\cdot)$ and $c(\cdot)$. Indeed, let us fix $t\in[0,1]$ and $x,y\in\R$. We observe that $|H(t,x,p)-H(t,y,p)|\leq k_R(t)\,(1+|p|)\,|x-y|$ for every $p\in\R$ if and only if $|\gamma(t,x)-\gamma(t,y)|\leq k_R(t)\,|x-y|$. Moreover, $|H(t,x,p)-H(t,x,q)|\leq c(t)\,(1+|x|)\,|p-q|$ for all $p,q\in\R$ if and only if $\gamma(t,x)\leq c(t)\,(1+|x|)$. Additionally, $|H(t,x,p)-H(t,x,q)|\leq C_R|p-q|$ for all $p,q\in\R$ if and only if $\gamma(t,x)\leq C_R$.
 Therefore, in view of Example \ref{ex-3}, we obtain our assertion. We notice that $L(t,x,\cdot\,)=H^{\ast}(t,x,\cdot\,)$ has the form
\begin{equation*}
L(t,x,v)=\left\{
\begin{array}{ccl}
+\infty & \tn{if} & |v|\;\geq\;\gamma(t,x)\,\not=\,0,\\[1mm]
\frac{\displaystyle |v|}{\displaystyle \gamma(t,x)-|v|} & \tn{if} &|v|\;<\,\;\gamma(t,x)\,\not=\,0, \\[2mm]
0 & \tn{if} & v=0,\;\, \gamma(t,x)=0,\\[0mm]
+\infty & \tn{if} & v\not=0,\;\,\gamma(t,x)=0.
\end{array}
\right.
\end{equation*}
The function $v\rightarrow L(t,x,v)$ is not bounded on  $\D L(t,x,\cdot)=(-\gamma(t,x),\gamma(t,x))$ for $\gamma(t,x)\not=0$.  So this Hamiltonian  does not satisfy the condition (A).
\end{Ex}

\begin{Ex}\label{ex-5}
Let us consider  functions $f$ and $l$ satisfying the following conditions:
\begin{enumerate}[leftmargin=9.7mm]
\item[$\pmb{(1)}$] $f:[0,T]\times\R^{\scriptscriptstyle N}\times\R^{\scriptscriptstyle M}\rightarrow\R^{\scriptscriptstyle N}$ and $l:[0,T]\times\R^{\scriptscriptstyle N}\times\R^{\scriptscriptstyle M}\rightarrow\R$ are continuous;
\item[$\pmb{(2)}$] for any $R\geq 0$ there exists an integrable function $k_R:[0,T]\to[0,+\infty)$ such that\\ $|f(t,x,a)-f(t,y,a)|+|l(t,x,a)-l(t,y,a)|\leq k_R(t)\,|x-y|$ for every $x,y\!\in\!\B_R$, $a\!\in\!\R^{\scriptscriptstyle M}$\\ and almost all $t\in[0,T]$;
\item[$\pmb{(3)}$] for any $R\geq 0$ there exists a constant $C_R\geq 0$ such that $|f(t,x,a)|\leq C_R$ for every\\ $t\in[0,T]$, $x\in\B_R$, $a\in\R^{\scriptscriptstyle M}$;
\item[$\pmb{(4)}$] there exists an integrable function $c:[0,T]\to[0,+\infty)$ such that $ c(t)\,(1+|x|)\geq$\\ $|f(t,x,a)|$ for every $x\in\R^{\scriptscriptstyle N}$, $a\in\R^{\scriptscriptstyle M}$ and almost all $t\in[0,T]$;
\item[$\pmb{(5)}$] $\lim_{\scriptstyle |a|\,\to\,\infty}[\inf_{\scriptstyle(t,x)\,\in\,[0,T]\times\,\B_R}l(t,x,a)]=+\infty$\, for every\, $R\geq 0$;
\item[$\pmb{(6)}$] $\{\,(f(t,x,a),l(t,x,a)+r)\mid a\in\R^{\scriptscriptstyle M}\!\!,r\in[0,\infty)\,\}$ is convex for all $t\in[0,T]$, $x\in\R^{\scriptscriptstyle N}\!$.
\end{enumerate}
 For instance, the following functions:
\begin{equation*}
\hat{f}(x,a_1,a_2)=a_1|x|/(1+|a_1|), \qquad \hat{l}(x,a_1,a_2)=|a_1|+|a_2|+|xa_2|/(1+|a_2|),
\end{equation*}
where $x\in\R$ and $(a_1,a_2)\in\R\times\R$, satisfy (1)-(6). 
Let $\gamma(\cdot,\cdot)$ be defined as in Example \ref{ex-3}. 

\pagebreak
\noindent Then the following functions:
\begin{equation*}
\check{f}(t,x,a_1,a_2)=a_1\gamma(t,x)/(1+|a_1|), \qquad \check{l}(t,x,a_1,a_2)=|a_1|+|a_2|+\gamma(t,x)|a_2|/(1+|a_2|),
\end{equation*}
where $t\in[0,1]$, $x\in\R$ and $(a_1,a_2)\in\R\times\R$, also satisfy (1)-(6), but with the unbounded functions $k_R(\cdot)$ and $c(\cdot)$.

We observe that if $f$ and $l$ satisfy (1)-(5), then the Hamiltonian $H$ given by 
\begin{equation}\label{hfl}
H(t,x,p)=  \sup\nolimits_{a\,\in\,\R^{\scriptscriptstyle M}}\,\{\,\langle\, p\,,f(t,x,a)\,\rangle\,-\,l(t,x,a)\,\}
\end{equation}
satisfies  (H1)-(H5). Moreover,  the Hamiltonian $\hat{H}$ given by \eqref{hfl} with $\hat{f},\hat{l}$ is the same as in Example \ref{ex-1} and the Hamiltonian $\check{H}$ given by \eqref{hfl} with $\check{f},\check{l}$ is the same as in Example~\ref{ex-4}.

Using the results from \cite[Sect. 4]{RTR73} and \cite[Sect. 2]{RTR75} one can show that if $f$ and $l$ satisfy (1)-(6), and $g$ is proper and lower semicontinuous, and $H$ is given by \eqref{hfl}, then
\begin{eqnarray*}
V(t_0,x_0) &=& \min_{\begin{array}{c}
\scriptstyle x(\cdot)\,\in\,\mathcal{A}\left([t_0,T],\R^{\scriptscriptstyle N}\right)\\[-1mm]
\scriptstyle x(t_0)=x_0
\end{array}}\!\!\big\{\,g(x(T))+\int_{t_0}^TH^{\ast}(t,x(t),\dot{x}(t))\,dt\,\big\}\\
&=& \min_{(x,a)(\cdot)\,\in\, \emph{S}_f(t_0,x_0)}\,\big\{\,g(x(T))+\int_{t_0}^Tl(t,x(t),a(t))\,dt\,\big\}.
\end{eqnarray*}
\end{Ex}
\noindent where $\emph{S}_f(t_0,x_0)$ denotes a set of all trajectory-measurable pairs of the control system
\begin{equation*}
\left\{\begin{array}{ll}
\dot{x}(t)=f(t,x(t),a(t)),& a(t)\in\R^{\scriptscriptstyle M}\!\!,\;\;\;\mathrm{a.e.}
\;\;t\in[t_0,T],\\
x(t_0)=x_0.&
\end{array}\right.
\end{equation*}

\begin{Rem}\label{rempp}
We cannot apply Theorem \ref{eauhje} to  $H:[0,1]\times\R\to\R$ given by the formula
\begin{equation*}
H(t,p):=\left\{
\begin{array}{ccl}
\max\left\{\,\frac{|p|}{\sqrt{t}}-\frac{1}{t},\,0\,\right\} & \tn{if} & p\in\R,\;t\not=0, \\[1mm]
0 & \tn{if} & p\in\R,\;t=0.
\end{array}
\right.
\end{equation*}
Because this Hamiltonian satisfies the conditions  (H1)-(H2) and (H4)-(H5), but it does not satisfy the condition (H3). 
\end{Rem}

\section{Reduction Theorem}\label{section-3}
\noindent Let us define the Hamiltonian $H:\R\times\R\rightarrow\R$ by the formula
\begin{equation*}
H(x,p)=\max\{\,|p|\,|x|-1,0\,\}.
\end{equation*}
This Hamiltonian satisfies conditions (H1)-(H5). Moreover, $L(x,\cdot\,)=H^{\ast}(x,\cdot\,)$ has the form
\begin{equation*}
L(x,v)=\left\{
\begin{array}{ccl}
+\infty & \tn{if} & v\not\in[-|x|,|x|\,],\;x\not=0,\\[1mm]
\left|\frac{\displaystyle v}{\displaystyle x}\right| & \tn{if} & v\in[-|x|,|x|\,],\;x\not=0, \\[1mm]
0 & \tn{if} & v=0,\; x=0,\\[0mm]
+\infty & \tn{if} & v\not=0,\;x=0.
\end{array}
\right.
\end{equation*}
Let $\lambda(x)=|x|+1$ for every $x\in\R$. Then $\lambda$ is Lipschitz continuous with a sublinear growth.\linebreak  Observe that $\|\D L(x,\cdot)\|\leq \lambda(x)$  and $\|L(x,\D L(x,\cdot))\|\leq\lambda(x)$ for all $x\in\R$. Therefore the above Hamiltonian satisfies the condition (A). 
The question  is how to construct the Hamiltonian $\bar{H}$ satisfying  \eqref{phc} and $(\bar{\tn{H}}1)$-$(\bar{\tn{H}}5)$. 

Barron-Jensen in \cite[Prop. 3.7]{B-J} proposed the following construction of $\bar{H}$:
\begin{eqnarray*}
\bar{H}(x,r,p,q) &:=& \sup_{v\in\D L(x,\cdot)}\,\{\,\langle v,p\rangle+q\,L(x,v)\,\}\\[0mm]
&=&\left\{
\begin{array}{ccl}
\max\left\{\,|p|\,|x|+q,0\,\right\} & \tn{if} & x\not=0\;\;\tn{and}\;\;r,p,q\in\R, \\[0mm]
0 & \tn{if} & x=0\;\;\tn{and}\;\;r,p,q\in\R.
\end{array}
\right.
\end{eqnarray*}
Observe that $\bar{H}$ satisfies \eqref{phc}, but the function $x\to\bar{H}(x,r,p,q)$ is not continuous  for all $(r,p,q)\in\R\times\R\times(0,\infty)$. Therefore $\bar{H}$ does not satisfy $(\bar{\tn{H}}1)$ and $(\bar{\tn{H}}5)$. It means that the construction of  $\bar {H}$ proposed by Barron-Jensen is not appropriate in this case.

\vspace{2mm}
Our construction of the Hamiltonian $\bar{H}$ is based on representations of $H$. The triple $(A,f,l)$ is called a representation of $H$ if it satisfies the following equality
\begin{equation*}
H(t,x,p)=  \sup\nolimits_{a\,\in\,A}\,\{\,\langle\, p\,,f(t,x,a)\,\rangle\,-\,l(t,x,a)\,\}.
\end{equation*}
We observe that the triple $A=[-1,1]$, $f(x,a)=a|x|$, $l(x,a)=|a|$ is a representation of the Hamiltonian from the example above. Moreover, the following Hamiltonian
\begin{eqnarray*}
\bar{H}(x,r,p,q) &:=& \sup\nolimits_{a\in A}\left\{\,\langle\,p,f(x,a)\,\rangle+q\,l(x,a)\,\right\}\\[0mm]
&=&\max\left\{\,|p|\,|x|+q,0\,\right\}
\end{eqnarray*}
satisfies \eqref{phc} and $(\bar{\tn{H}}1)$-$(\bar{\tn{H}}5)$. Obviously,  our construction of the Hamiltonian $\bar {H}$ makes sense, provided that we can find an appropriately regular representation of $H$. In papers \cite{AM,AM1} one proved that this kind representations always exist.

\vspace{2mm}
Now we explain the reason that our construction of Hamiltonian $\bar{H}$ gives the expected results
in contrast to the construction of Barron-Jensen. We know that Hamiltonian $H$ from the above example
satisfies (H1)-(H5). Therefore, by Theorem \ref{tw2_rlhmh}  Lagrangian $L$ satisfies (L6). It means that $L$ is lower-Lipschitz continuous. On the other hand, our example shows that  $L$ is not upper-Lipschitz continuous and, what is more, it is not upper semicontinuous. Indeed, $L$ is not upper semicontinuous, because $\limsup_{i\rightarrow\infty}L\left(1/i,1/i\right)=1\nleqslant 0=L(0,0)$. It is not difficult to see that the main reason for ineffectiveness of the construction of $\bar{H}$ proposed by Barron-Jensen  is the lack of upper-Lipschitz continuity of $L$. If the triple $(A,f,l)$ is a faithful representation of $H$, then the functions $f$ and $l$ are Lipschitz continuous. In particular, the function $l$ is lower/upper-Lipschitz continuous. Due to that, our construction of Hamiltonian $\bar{H}$ gives the expected result.

\vspace{2mm}
In the proof of  Theorem \ref{thmrd} we need slightly modified Theorems 3.1 and 3.4 from \cite{AM}.
We use the first theorem to prove $\tn{(B)}\Rightarrow\tn{(A)}$ and the second theorem to prove $\tn{(A)}\Rightarrow\tn{(B)}$.

\begin{Th}\label{rtwcs1}
Assume that $H$ satisfies the condition \tn{(A)} from Theorem \ref{thmrd}. Then there exists a representation $(\B,f,l)$ of $H$ such that $\B$ is a closed unit ball in $\R^{\scriptscriptstyle N+1}$ and 
\begin{enumerate}[leftmargin=9.7mm]
\item[$\pmb{(\tn{R1})}$] $f:[0,T]\times\R^{\scriptscriptstyle N}\times\B\rightarrow\R^{\scriptscriptstyle N}$ and $l:[0,T]\times\R^{\scriptscriptstyle N}\times\B\rightarrow\R$ are continuous;
\item[$\pmb{(\tn{R2})}$] for any $R\geq 0$ there exists an integrable function $K_R:[0,T]\to[0,+\infty)$ such that\\ $|f(t,x,a)-f(t,y,a)|+|l(t,x,a)-l(t,y,a)|\leq K_R(t)\,|x-y|$ for every $x,y\!\in\!\B_R$, $a\!\in\!\B$\\ and almost all $t\in[0,T]$\tn{;}
\item[$\pmb{(\tn{R3})}$] $|f(t,x,a)|+|l(t,x,a)|\leq C(t)(1+|x|)$ for all $x\in\R^{\scriptscriptstyle N}$, $a\in\B$, and for a.e. $t\in[0,T]$,\\ and some integrable function $C:[0,T]\to[0,+\infty)$.
\end{enumerate}
\end{Th}
\begin{proof}
Assume that $H$ satisfies (H1)-(H5) with $c(\cdot)$, $k_R(\cdot)$. Let $\lambda$ be as in the condition (A) with $\vartheta(\cdot)$, $\zeta_R(\cdot)$. We define $\mathrm{e}:[0,T]\times\R^{\scriptscriptstyle N}\times\R^{\scriptscriptstyle N+1}\to\R$ by the formula
$$\mathrm{e}(t,x,a):=S_{\!\scriptscriptstyle N+1}\big[\E H^{\ast}(t,x,\cdot)\cap\B\big(\omega(t,x)\,a,2\,\dist(\omega(t,x)\,a,\E H^{\ast}(t,x,\cdot))\big)\big],$$
where $S_{\!\scriptscriptstyle N+1}[\,\cdot\,]$ in the Steiner selection and $\omega(t,x):=2\lambda(t,x)+1$. By \cite[Section 5]{AM} the function $\mathrm{e}(\cdot,\cdot,\cdot)$ is well defined and continuous. Moreover, $\mathrm{e}(t,\,\cdot\,,a)$ is $\frac{1}{2}K_R(t)$-Lipschitz on $\B_R$ for a.e. $t\in[0,T]$ and all $a\in\B$, where $K_R(t):=20(N+1)(k_R(t)+2\zeta_R(t))$. Additionally, 
\begin{equation}\label{inrep}
\forall\,(t,x)\in[0,T]\times\R^{\scriptscriptstyle N}\;\;\;\G H^{\ast}(t,x,\cdot)\subset\mathrm{e}(t,x,\B)\subset\E H^{\ast}(t,x,\cdot).
\end{equation}

We show that $|\mathrm{e}(t,x,a)|\leq\frac{1}{2}C(t)(1+|x|)$ for a.e. $t\in[0,T]$ and all $x\in\R^{\scriptscriptstyle N}$, $a\in\B$, where $C(t):=20\vartheta(t)+6$. Indeed, by \cite[p. 366]{A-F} we have $S_{\!\scriptscriptstyle N+1}[W]\in W$ for every nonempty, convex and compact subset $W$ of $\R^{\scriptscriptstyle N+1}$. Thus, for all $t\in[0,T]$, $x\in\R^{\scriptscriptstyle N}$, $a\in\R^{\scriptscriptstyle N+1}$, we obtain
\begin{equation}\label{inrb1}
\mathrm{e}(t,x,a)\in\B\big(\omega(t,x)\,a,2\,\dist(\omega(t,x)\,a,\E H^{\ast}(t,x,\cdot))\big).
\end{equation}
Let $v\in\D H^{\ast}(t,x,\cdot)\neq\emptyset$. Then $(v,H^{\ast}(t,x,v))\in\E H^{\ast}(t,x,\cdot)$. Hence, we obtain 
\begin{equation}\label{inrb2}
\dist(0,\E H^{\ast}(t,x,\cdot))\leq |(v,H^{\ast}(t,x,v))|\leq |v|+|H^{\ast}(t,x,v))|\leq 2\lambda(t,x)\end{equation}
for all $t\in[0,T]$ and $x\in\R^{\scriptscriptstyle N}$.
 Combining \eqref{inrb1} and \eqref{inrb2} we obtain
\begin{eqnarray*}
|\mathrm{e}(t,x,a)| &\leq & \omega(t,x)\,|a| +2\,\dist(\omega(t,x)\,a,\E H^{\ast}(t,x,\cdot))\\
&\leq & 3\,\omega(t,x)+2\,\dist(0,\E H^{\ast}(t,x,\cdot))\\
&\leq & 3\,\omega(t,x)+4\,\lambda(t,x)\;\;=\;\;10\,\lambda(t,x)+3\\
&\leq & (10\vartheta(t)+3)(1+|x|)
\end{eqnarray*}
for all $t\in[0,T]$, $x\in\R^{\scriptscriptstyle N}$, $a\in\B$.

 Next, we define the functions $f$ and $l$ as components of the function~$\mathrm{e}$, i.e., $\mathrm{e}=(f,l)$. Then, by \eqref{inrep} and \cite[Prop. 5.7]{AM}, the triple $(\B,f,l)$ is a representation of $H$.  It is not difficult to prove that the functions $f$ and $l$ satisfy (R1)-(R3). 
\end{proof}

\begin{Th}\label{rtwcs2}
If the triple $(\B,f,l)$ is a representation of $H$ and the functions $f,\,l$ satisfy \tn{(R1)-(R3)}, then $H$ satisfies the condition \tn{(A)}. 
\end{Th}
\begin{proof}
Let  $(\B,f,l)$ be a representation of $H$ and $f,\,l$ satisfy \tn{(R1)-(R3)} with $C(\cdot)$, $K_R(\cdot)$. Using (R1)-(R3) one can show that $H$ satisfies (H1)-(H5).
We define a simplex in $\R^{\scriptscriptstyle N+1}$~by 
$$\Delta:= \{(\alpha_0,\dots,\alpha_{\scriptscriptstyle N})\in[0,1]^{\scriptscriptstyle N+1}\mid\alpha_0+\dots+\alpha_{\scriptscriptstyle N}=1\}.$$
Obviously, the set $\Delta$ is compact. Moreover, we define the set $\,\mathbbmtt{A}\,$
by $\mathbbmtt{A}:=\B^{\,\scriptscriptstyle N+1}\!\times\Delta$. We notice that the set $\mathbbmtt{A}$
is compact. The functions $\mathbbmtt{f}$, $\mathbbmtt{l}$ are defined for every $t\in[0,T]$, $x\in\R^{\scriptscriptstyle N}$ and $\mathbbmtt{a}=(a_0,\dots,a_{\scriptscriptstyle N},\alpha_0,\dots,\alpha_{\scriptscriptstyle N})\in \B^{\,\scriptscriptstyle N+1}\!\times\Delta=\mathbbmtt{A}$ by the formulas:
$$\mathbbmtt{f}(t,x,\mathbbmtt{a}):= \sum_{n=0}^{\scriptscriptstyle N}\alpha_n f(t,x,a_n),\qquad \mathbbmtt{l}(t,x,\mathbbmtt{a}):= \sum_{n=0}^{\scriptscriptstyle N}\alpha_n l(t,x,a_n).$$
Using (R1)-(R3) one can show that $\mathbbmtt{f}$, $\mathbbmtt{l}$ are continuous. Moreover, $\mathbbmtt{f}(t,\cdot,\mathbbmtt{a})$ and $\mathbbmtt{l}(t,\cdot,\mathbbmtt{a})$ are $K_R(t)$-Lipschitz on $\B_R$ for a.e.  $\!t\in[0,T]$ and all $\mathbbmtt{a}\in \mathbbmtt{A}$. Additionally, $|\mathbbmtt{f}(t,x,\mathbbmtt{a})|\leq C(t)(1+|x|)$ and $|\mathbbmtt{l}(t,x,\mathbbmtt{a})|\leq C(t)(1+|x|)$ for a.e.  $\!t\in[0,T]$ and all $x\in\R^{\scriptscriptstyle N}$, $\mathbbmtt{a}\in \mathbbmtt{A}$. It is not difficult to show that the triple  $(\mathbbmtt{A},\mathbbmtt{f},\mathbbmtt{l})$ is also the representation of $H$. In view of \cite[Lem. 4.2]{AM} and \cite[Thm. 2.29]{R-W} we obtain the following equalities:
\begin{equation}\label{odeq}
\mathbbmtt{f}(t,x,\mathbbmtt{A})= \mathrm{conv}f(t,x,A)=\D H^{\ast}(t,x,\cdot).
\end{equation}
We define $\lambda:[0,T]\times\R^{\scriptscriptstyle N}\to[0,\infty)$ by the formula
$$\lambda(t,x):=\sup\nolimits_{a\in\mathbbmtt{A}}\big\{\,|\mathbbmtt{f}(t,x,\mathbbmtt{a})|+|\mathbbmtt{l}(t,x,\mathbbmtt{a})|\,\big\}.$$
We observe that $\lambda$ is a continuous function. Moreover, $\lambda(t,\cdot)$ is $\zeta_R(t)$-Lipschitz on $\B_R$ for a.e.  $\!t\in[0,T]$, where $\zeta_R(t):=2K_R(t)$. Additionally, $\lambda(t,x)\leq\vartheta(t)(1+|x|)$ for a.e. $t\in[0,T]$ and all $x\in\R^{\scriptscriptstyle N}$, where $\vartheta(t):=2C(t)$. Fix $t\in[0,T]$ and $x\in\R^{\scriptscriptstyle N}$. If $\bar{v}\in\D H^{\ast}(t,x,\cdot)$, then by the equality \eqref{odeq} there exists $\bar{\mathbbmtt{a}}\in \mathbbmtt{A}$ such that   $\bar{v}=\mathbbmtt{f}(t,x,\bar{\mathbbmtt{a}})$. Therefore, by \cite[Lem. 4.1]{AM},
\begin{equation}\label{wara1}
H^{\ast}(t,x,\bar{v})=H^{\ast}(t,x,\mathbbmtt{f}(t,x,\bar{\mathbbmtt{a}}))\leq \mathbbmtt{l}(t,x,\bar{\mathbbmtt{a}})\leq \lambda(t,x).
\end{equation}
Since $(\mathbbmtt{A},\mathbbmtt{f},\mathbbmtt{l})$ is a representation of $H$, there exists $\tilde{\mathbbmtt{a}}\in\mathbbmtt{A}$ such that $H(t,x,0)=-\mathbbmtt{l}(t,x,\tilde{\mathbbmtt{a}})$
The latter equality implies, for all $v\in\D H^{\ast}(t,x,\cdot)$,
\begin{equation}\label{wara2}
-\lambda(t,x)\leq -|\mathbbmtt{l}(t,x,\tilde{\mathbbmtt{a}})|=-|H(t,x,0)|\leq H^{\ast}(t,x,v).
\end{equation}
Combining \eqref{wara1} and \eqref{wara2} we obtain $\|H^{\ast}(t,x,\D H^{\ast}(t,x,\cdot))\|\leq\lambda(t,x)$ for all  $t\in[0,T]$, $x\in\R^{\scriptscriptstyle N}$. Moreover, using \eqref{odeq}, we obtain $\|\D H^{\ast}(t,x,\cdot)\|\leq\lambda(t,x)$ for all  $t\in[0,T]$, $x\in\R^{\scriptscriptstyle N}$.  It completes the proof.
\end{proof}

\begin{proof}[Proof Theorem \ref{thmrd}]
Let $H:[0,T]\times\R^{\scriptscriptstyle N}\times\R^{\scriptscriptstyle N}\rightarrow\R$ be given.

Suppose that $H$ satisfies the condition (A). Then, in view of Theorem \ref{rtwcs1}, there exists a representation $(\B,f,l)$ of $H$ satisfying (R1)-(R3). Define $\bar{H}:[0,T]\times\R^{\scriptscriptstyle N+1}\!\times\R^{\scriptscriptstyle N+1}\to\R$~by
\begin{equation*}
\bar{H}(t,x,r,p,q):=\sup\nolimits_{a\in\B}\big\langle\big(p,q\big),\big(f(t,x,a),l(t,x,a)\big)\big\rangle.
\end{equation*}
Using (R1)-(R3), it is not difficult to prove that $\bar{H}$ satisfies $(\bar{\tn{H}}1)$-$(\bar{\tn{H}}5)$.
Moreover, $\bar{H}$ is positively homogeneous in $(p,q)$, this comes directly from its definition. Since the triple $(\B,f,l)$ is a representation of $H$, we have 
\begin{eqnarray*}
\bar{H}(t,x,r,p,-1) &=& \sup\nolimits_{a\in\B}\big\langle\big(p,-1\big),\big(f(t,x,a),l(t,x,a)\big)\big\rangle\\
&=& \sup\nolimits_{a\in\B}\big\{\big\langle p,f(t,x,a)\big\rangle-l(t,x,a)\big\}\;\;=\;\;H(t,x,p).
\end{eqnarray*}
Thus $\bar{H}$ satisfies \eqref{phc}. Consequently, the condition (B) holds.

Conversely, suppose that the condition (B) holds. Then there exists  $\bar{H}$ satisfying \eqref{phc} and $(\bar{\tn{H}}1)$-$(\bar{\tn{H}}5)$ with $c(\cdot)$, $k_R(\cdot)$. We define the set-valued map $\bar{E}:[0,T]\times\R^{\scriptscriptstyle N+1}\multimap\R^{\scriptscriptstyle N+1}$ by 
\begin{equation*}
\bar{E}(t,x,r):=\big\{(v,\eta)\in\R^{\scriptscriptstyle N+1}\;\big|\;\big\langle(v,\eta),(p,q)\big\rangle\leq \bar{H}(t,x,r,p,q)\;\;\tn{for all}\;\;(p,q)\in\R^{\scriptscriptstyle N+1}\big\}.
\end{equation*}
From results in \cite[Section 7]{HF} we deduce that $\bar{E}$ has nonempty, compact, convex values and is continuous. Moreover,  $\bar{E}(t,\cdot,\cdot)$ is $k_R(t)$-Lipschitz on $\B_R$ for a.e. $t\in[0,T]$ and $\|\bar{E}((t,x,r))\|\leq c(t)(1+|(x,r)|)$ for a.e. $t\in[0,T]$ and all $(x,r)\in\R^{\scriptscriptstyle N+1}$. Additionally,
\begin{equation*}
\bar{H}(t,x,r,p,q)=\sup\nolimits_{\;(v,\eta)\in \bar{E}(t,x,r)}\big\langle(p,q),(v,\eta)\big\rangle
\end{equation*}
for all $t\in[0,T]$, $(x,r)\in\R^{\scriptscriptstyle N+1}$, $(p,q)\in\R^{\scriptscriptstyle N+1}$.
Let $K_R(t):=20(N+1)k_R(t)$ and $C(t):=2c(t)$. By \cite[Thm. 9.6.2]{A-F} there exists a continuous  function $\bar{\mathrm{e}}:[0,T]\times\R^{\scriptscriptstyle N+1}\!\times\B\to\R^{\scriptscriptstyle N+1}$ such that  $\bar{\mathrm{e}}(t,x,r,\B)=\bar{E}(t,x,r)$ for all  $t\in[0,T]$, $(x,r)\in\R^{\scriptscriptstyle N+1}$. Moreover, the function $\bar{\mathrm{e}}(t,\cdot,\cdot,a)$ is $\frac{1}{2}K_R(t)$-Lipschitz on $\B_R$ and $|\bar{\mathrm{e}}(t,x,r,a)|\leq \frac{1}{2}C(t)(1+|(x,r)|)$ for a.e. $t\in[0,T]$ and all $(x,r)\in\R^{\scriptscriptstyle N+1}$, $a\in\B$. Next, we define functions $\bar{f}$ and $\bar{l}$ as components of a function~$\bar{\mathrm{e}}$, i.e., $\bar{\mathrm{e}}=(\bar{f},\bar{l})$. Then we have
\begin{eqnarray*}
H(t,x,p) &=& \bar{H}(t,x,0,p,-1)\\
&=& \sup\nolimits_{(v,\eta)\in \bar{E}(t,x,0)}\big\langle(p,-1),(v,\eta)\big\rangle\\
&=& \sup\nolimits_{a\in\B}\big\langle\big(p,-1\big),\big(\bar{f}(t,x,0,a),\bar{l}(t,x,0,a)\big)\big\rangle\\
&=& \sup\nolimits_{a\in\B}\big\{\big\langle p\,,\bar{f}(t,x,0,a)\big\rangle\,-\,\bar{l}(t,x,0,a)\big\}.
\end{eqnarray*}
Let $f(t,x,a):=\bar{f}(t,x,0,a)$ and $l(t,x,a):=\bar{l}(t,x,0,a)$. Then $(\B,f,l)$  is a representation of $H$ and the functions $f,\,l$ satisfy (R1)-(R3). Therefore, in view of Theorem \ref{rtwcs2}, $H$ satisfies the condition \tn{(A)}, which completes the proof.
\end{proof}

\section{Viability and Invariance Theorems}\label{section-4}

\noindent In this section we present viability and invariance theorems with unbounded differential inclusions. Working with unbounded differential inclusions we encounter new problems  which we solve in this section. Let $\pi_K(\cdot)$ be a projection of  $\R^{\scriptscriptstyle N}$ onto a nonempty closed convex subset  $K$ of $\R^{\scriptscriptstyle N}$. We denote by $\chi_K(\cdot)$ an indicator function of a  subset $K$ of $\R^{\scriptscriptstyle N}$.

\subsection{Invariance Theorem} We start by formulating the invariance theorem:
\begin{Th}[Invariance Theorem]\label{Twon}
Assume that $L$ satisfies \tn{(L1)-(L6)}. Let $U$ be a proper and lower semicontinuous function satisfying the following condition:
\begin{equation}\label{mwn}\begin{split}
& \tn{For every}\;(t,x)\in\D U\cap (0,T]\times\R^{\scriptscriptstyle N},\;\tn{every}\;(n^t,n^x,n^u)\in N_{\E U}(t,x,U(t,x)),\\[-2mm]
& \tn{every}\;v\in \D L(t,x,\cdot),\;\tn{one has}\;n^t+\langle v,n^x \rangle-n^uL(t,x,v)\geq 0.
\end{split}\end{equation}
Then for every $t_0\in[0,T)$ and every absolutely continuous function $(x,u):[t_0,T]\rightarrow\R^{\scriptscriptstyle N+1}$ satisfying $(\dot{x},\dot{u})(t)\in Q(t,x(t))$  for a.e. $t\in[t_0,T]$ and $u(T)\geq U(T,x(T))$, we obtain the following inequality $u(t)\geq U(t,x(t))$ for all $t\in[t_0,T]$. 
\end{Th}

We show that the standard methods  fail when used in trying to prove the Theorem \ref{Twon}. 

\vspace{2mm}
In the first method we take $\hat{Q}(s,x,u)=\{1\}\times Q(\pi_{[0,T]}(s),\pi_{\B_R}(x))$ for all $s,u\in\R$, $x\in\R^{\scriptscriptstyle N}$, where $R=\|x([t_0,T])\|+1$. Define $y(t)=(t,x(t),u(t))$ and  $\varphi(t)=\dist(y(t),\E U)$ for all $t\in[t_0,T]$. We notice that  $\dot{y}(t)\in\hat{Q}(y(t))$ for a.e. $t\in[t_0,T]$ and $\varphi$ is absolutely continuous with $\varphi(T)=0$. The latter and \eqref{mwn} imply $\dot{\varphi}(t)\leq k\,\varphi(t)$ for a.e. $t\in[t_0,T]$, provided that $\hat{Q}$ is $k$-Lipschitz with respect to all variables; see \cite[Thm. 4.2]{P-Q}. Hence, using Gronwall's lemma, we have $\varphi(t)=0$ for all $t\in[t_0,T]$. Therefore, $u(t)\geq U(t,x(t))$ for all $t\in[t_0,T]$.  We cannot apply this method to prove Theorem \ref{Twon}, because we do not assume that Q is Lipschitz continuous with respect to the time variable. 

\vspace{2mm}
The second method is similar to the first one, but does not require the assumption that Q is Lipschitz continuous with respect to the time variable. Let $\hat{Q}(t,x,u)=Q(t,\pi_{\B_R}(x))$ for all $t\in[0,T]$,  $x\in\R^{\scriptscriptstyle N}$, $u\in\R$, where $R=\|x([t_0,T])\|+1$. Then $(\dot{x},\dot{u})(t)\in \hat{Q}(t,x(t),u(t))$  and $\hat{Q}(t,\cdot,\cdot)$ is $k_R(t)$-Lipschitz for a.e. $t\in[0,T]$. Define $\varphi(t)=\dist((x,u)(t),\E U(t,\cdot))$ for all\linebreak $t\in[t_0,T]$. Our assumptions imply that $\varphi(\cdot)$ is lower semicontinuous with $\varphi(T)=0$.\linebreak One can prove, due to \eqref{mwn}, that $d\varphi(t)(1)\leq k(t)\,\varphi(t)$ for a.e. $t\in[t_0,T]$; see \cite[p. 394]{IPA}. The latter inequality yields $\varphi\equiv 0$, provided that $d\varphi(t)(1)<\infty$  for every $t\in[t_0,T]$; see\linebreak \cite[p. 42]{H-T}. Unfortunately, without additional assumptions, we cannot say whether the latter condition is true. So, we cannot apply this method to prove Theorem \ref{Twon}.

\vspace{2mm}
The third method is based on the reduction set-valued map with unbounded values to the case of set-valued map with compact values. This kind of method was used by author in the work \cite{AM2}. Let $R=\|x([t_0,T])\|+1$. By (L3) there exists $\delta>0$ such that $Q(t,x)\cap(\B_\delta\times[-\delta,\delta])\not=\emptyset$ for all $t\in[0,T]$, $x\in\B_R$. Let  $\lambda=1+3\delta+\textrm{ess sup}_{\;t\in[t_0,T]\;}|(\dot{x},\dot{u})(t)|$. We define $\hat{Q}(t,x,u)=Q(t,\pi_{\B_R}(x))\cap (\B_\lambda\times[-\lambda,\lambda])$ for every $t\in[0,T]$,  $x\in\R^{\scriptscriptstyle N}$, $u\in\R$.  The set-valued map $\hat{Q}$ is continuous  and $\hat{Q}(t,\cdot,\cdot)$ is $k_R(t)$-Lipschitz for a.e. $t\in[0,T]$; see \cite[Prop. 4.39]{R-W}. Moreover, $(\dot{x},\dot{u})(t)\in \hat{Q}(t,x(t),u(t))$ for a.e. $t\in[0,T]$.
Now we can use invariance theorem with compact differential inclusions proved in the work of Frankowska  \cite[Thm. 3.3]{HF}. Unfortunately, we cannot use this reduction to prove our Theorem \ref{Twon}, because in general the function $\dot{u}(\cdot)$ is unbounded, i.e. $\lambda=\infty$.

\vspace{2mm}
As we mentioned earlier, Frankowska proved the invariance theorem for set-valued maps with compact values; see \cite[Thm. 3.3]{HF}. In the proof of this theorem, she used the parametrization $\e$ of the set-valued map $\hat{Q}$ with the parameter set $\B$. Indeed, in view of \cite[Thm. 9.6.2]{A-F} there exists a continuous  single-valued map $\e$ defined on $[0,T]\times\R^{\scriptscriptstyle N+1}\times\B$ into $\R^{\scriptscriptstyle N+1}$ such that  $\e(t,x,u,\B)=\hat{Q}(t,x,u)$ for all  $t\in[0,T]$, $x\in\R^{\scriptscriptstyle N}$, $u\in\R$. Moreover, $\e(t,\cdot,\cdot,\cdot)$ is $k_R(t)$-Lipschitz  for a.e. $t\in[0,T]$. Let $(\dot{x},\dot{u})(t)\in \hat{Q}(t,x(t),u(t))$ for a.e. $t\in[t_0,T]$. \linebreak
Then, by Filippov Theorem, there exists a measurable function $a:[t_0,T]\to\B$ such that\linebreak $(\dot{x},\dot{u})(t)=\e(t,x(t),u(t),a(t))$ for a.e. $t\in[t_0,T]$.  We notice that  $a(\cdot)$ is also integrable, since it is bounded. Then we can choose a sequence of continuous functions $a_i:[t_0,T]\to\B$ converging to $a(\cdot)$ in $L^1$-spaces. Let $f_i(t,x,u)=\e(t,x,u,a_i(t))$ for all  $t\in[t_0,T]$, $x\in\R^{\scriptscriptstyle N}$, $u\in\R$. We observe that $f_i$ is continuous and $f_i(t,\cdot,\cdot)$ is $k_R(t)$-Lipschitz  for a.e. $t\in[t_0,T]$.\linebreak By Picard-Lindel\"{o}f Theorem, there exist unique solutions $(\dot{x}_i,\dot{u}_i)(t)=f_i(t,x_i(t),u_i(t))$ with $(x_i,u_i)(T)=(x(T),u(T))$. In view of Viability  Theorem, we have $(x_i,u_i)(t)\in\E U(t,\cdot)$ for all $t\in[t_0,T]$. Passing to the limit as $i\to\infty$, we obtain $(x,u)(t)\in \E U(t,\cdot)$ for all $t\in[t_0,T]$. \linebreak
Obviously, we cannot use this method to prove Theorem \ref{Twon}, because parametrization of the set-valued map Q with the compact parameter set does not exist. However, there is a parametrization of the set-valued map Q with the unbounded parameter set:
\begin{Th}[\tn{\cite[Thm. 5.1]{AM1}}]\label{th-oparam}
Assume that  $Q:[0,T]\times\R^{\scriptscriptstyle N}\multimap\R^{\scriptscriptstyle M}$ satisfies  $(\mathcal{Q}1)$-$(\mathcal{Q}3)$. Then there exists a continuous function $\e:[0,T]\times\R^{\scriptscriptstyle N}\times\R^{\scriptscriptstyle M}\to\R^{\scriptscriptstyle M}$ such that
\begin{enumerate}
\item[\tn{\bf{(P1)}}] $\e(t,x,\R^{\scriptscriptstyle M})=Q(t,x)$ \;for all\,  $t\in[0,T]$, $x\in\R^{\scriptscriptstyle N}$\tn{;}\vspace{2mm}
\item[\tn{\bf{(P2)}}] $a=\e(t,x,a)$ \;for all\,  $a\in Q(t,x)$, $t\in[0,T]$, $x\in\R^{\scriptscriptstyle N}$. \vspace{2mm}
\item[]\hspace{-1.28cm}Additionally, if $Q$ satisfies $(\mathcal{Q}6)$, then we have\vspace{2mm}
\item[\tn{\bf{(P3)}}] $\big|\e(t,x,a)-\e(t,y,b)\big|\leq 10\,M\big(k_R(t)|x-y|+|a-b|\big)$  for all $x,y\in\B_R$, $a,b\in\R^{\scriptscriptstyle M}$ and almost all $t\in[0,T]$.
\end{enumerate}
\end{Th}
\noindent It turns out that the lack of compactness of the parameter set causes a serious problem. Indeed, let $(\dot{x},\dot{u})(t)\in Q(t,x(t))$ for a.e. $t\in[t_0,T]$. If $\e(t,x,\R^{\scriptscriptstyle N+1})=Q(t,x)$ for all $t\in[0,T]$, $x\in\R^{\scriptscriptstyle N}$, then  by Filippov Theorem there exists a measurable function $a:[t_0,T]\to\R^{\scriptscriptstyle N+1}$ such that $(\dot{x},\dot{u})(t)=\e(t,x(t),a(t))$ for a.e. $t\in[t_0,T]$. Obviously, the measurable function $a(\cdot)$ may not be integrable. Therefore, we cannot approximate  $a(\cdot)$  by continuous functions in $L^1$-spaces, which prevents the continuation of the Frankowska method. It turns out that the above approximation problem can be solved using an extra-property. Property (P2) in Theorem \ref{th-oparam} is called the extra-property. We discovered this property  by researching the regularities of the value functions in \cite{AM1}. We observe that Parametryzation Theorem 9.6.2 in the monograph of Aubin-Frankowska \cite{A-F} does not contain the property of type (P2). Now, we show how this extra-property can be used to solve the approximation problem. Let $(\dot{x},\dot{u})(t)\in Q(t,x(t))$ for a.e. $t\in[t_0,T]$. Define $a(t):=(\dot{x}(t),\dot{u}(t))$ for a.e. $t\in[t_0,T]$.\linebreak In view of (P2) we have $(\dot{x},\dot{u})(t)=a(t)=\e(t,x(t),a(t))$ for a.e. $t\in[t_0,T]$. Since $(x,u)(\cdot)$ is an absolutely continuous function, $(\dot{x},\dot{u})(\cdot)$ is an integrable function. Thus, $a(\cdot)$ is also integrable. Therefore, $a(\cdot)$ can be approximated by continuous functions in $L^1$-spaces.\linebreak Summarizing, the  method of Frankowska can be applied to the unbounded case, provided that we have appropriately regular  parametrization of the set-valued map.

\begin{proof}[Proof of Theorem \ref{Twon}]
In view of Corollary \ref{wrow-wm} and Theorem \ref{tw2_rlhmh} the set-valued map $Q$ satisfies the conditions of Theorem \ref{th-oparam}. Therefore there exists a continuous function $\e$ satisfying the assertions: (P1), (P2) and (P3) of Theorem \ref{th-oparam}. Set $R:=\|x(\cdot)\|+\|u(\cdot)\|+2$, where $\|\cdot\|$ denotes the supremum norm.  

We define the function $\bar{\e}:[0,T]\times\R^{\scriptscriptstyle N}\times\R\times\R^{\scriptscriptstyle N+1}\to\R^{\scriptscriptstyle N}\times\R$ by the formula
$$\bar{\e}(t,x,u,a):=\e(t,\pi_{\B_R}(x),a).$$
Since $\e$ satisfies (P3), we conclude that for all $t\in[0,T]$, $x,y\in\R^{\scriptscriptstyle N}$, $u,w\in\R$, $a,b\in\R^{\scriptscriptstyle N+1}$, 
\begin{eqnarray}
|\bar{\e}(t,x,u,a)-\bar{\e}(t,y,w,b)| &= & |\e(t,\pi_{\B_R}(x),a)-\e(t,\pi_{\B_R}(y),b)|\nonumber\\
&\leq & 10\,(N+1)\big(\,k_R(t)\,|\pi_{\B_R}(x)-\pi_{\B_R}(y)|+|a-b|\big)\nonumber\\
&\leq & 10\,(N+1)\big(\,k_R(t)\,|x-y|+|u-w|+|a-b|\big).\label{prowblt-4}
\end{eqnarray}
By (P1) we get $\bar{\e}(t,x,u,\R^{\scriptscriptstyle N+1})=Q(t,x)$ for all $t\in[0,T]$, $x\in\B_R$, $u\in\R$. We always have $N_{\E U}(t,x,u)\subset N_{\E U}(t,x,U(t,x))$ for all $(t,x,u)\in\E U$. Therefore, in view of \eqref{mwn}, 
\begin{equation*}
\begin{split}
& \forall\;(t,x,u)\in\E U\cap (0,T]\times\B_R\times\R,\;\;\forall\; (n^t,n^x,n^u)\in N_{\E U}(t,x,u),\\
& \forall\; a\in\R^{\scriptscriptstyle N+1}\!\!,\;\;n^t+\langle\, (n^x,n^u),\,\bar{\e}(t,x,u,a)\,\rangle\,\geq\, 0.
\end{split}
\end{equation*}
By continuity of $\bar{\e}$ and properties \cite[Cor. 6.21]{R-W}, \cite[p.130]{A-F} of  normal and tangent canes 
\begin{equation}\label{prowblt-02}
\forall\,(t,x,u)\in\E U\cap (0,T]\times\B_R\times\R,\;\forall\, a\in\R^{\scriptscriptstyle N+1}\!\!,\;(-1, -\bar{\e}(t,x,u,a))\in T_{\E U}(t,x,u).
\end{equation}
Set $a(\cdot):=(\dot{x},\dot{u})(\cdot)\in Q(\cdot,x(\cdot))$. Then $(\dot{x}(t),\dot{u}(t))=\bar{\e}(t,x(t),u(t),a(t))$ for a.e. $t\in[t_0,T]$. Indeed, since $\e$ satisfies (P2) and $x([t_0,T])\subset\B_R$, we have 
$$(\dot{x}(t),\dot{u}(t))\,=\,a(t)\,=\,\e(t,x(t),a(t))\,=\,\e(t,\pi_{\B_R}(x(t)),a(t))\,=\,\bar{\e}(t,x(t),u(t),a(t)).$$
Since $a(\cdot)\in L^1([t_0,T],\R^{\scriptscriptstyle N+1})$, we can choose functions  $a_i(\cdot)\in C([t_0,T],\R^{\scriptscriptstyle N+1})$ such that $\lim_{i\to\infty}\|(a_i-a)(\cdot)\|_{L^1}=0$, where $\|\cdot\|_{L^1}$ denotes the standard norm in   $L^1([t_0,T],\R^{\scriptscriptstyle N+1})$. Since  $\e$ satisfies \eqref{prowblt-4}, in view of Picard-Lindel\"{o}f Theorem, there exist unique solutions  $(x_i,u_i)(\cdot)\in C^1([t_0,T],\R^{\scriptscriptstyle 2N+1})$  of  initial value problems
\begin{align}
&(\dot{x}_i(t),\dot{u}_i(t))=\bar{\e}(t,x_i(t),u_i(t),a_i(t))\;\;\tn{for all}\;\;t\in[t_0,T],\label{pcpc}\\
&(x_i,u_i)(T)=(x(T),u(T)).\nonumber
\end{align}
By the inequality \eqref{prowblt-4} we have
\begin{eqnarray*}
|(\dot{x}_i,\dot{u}_i)(t)-(\dot{x},\dot{u})(t)|&=& |\bar{\e}(t,x_i(t),u_i(t),a_i(t))-\bar{\e}(t,x(t),u(t),a(t))|\\
&\leq & 10\,(N+1)\big(\,k_R(t)\,|x_i(t)-x(t)|+|u_i(t)-u(t)|+|a_i(t)-a(t)|\big).
\end{eqnarray*}
Therefore, because of Gronwall’s lemma,
\begin{equation*}
\|(x_i,u_i)(\cdot)-(x,u)(\cdot)\|\leq 10\,(N+1)\,\|(a_i-a)(\cdot)\|_{L^1}\,\exp\left(\int_{t_0}^T20\,(N+1)\,(1+k_R(t))\;dt\right)\!.
\end{equation*} 
Since $\lim_{i\to\infty}\|(a_i-a)(\cdot)\|_{L^1}=0$, we have $\lim_{i\to\infty}\|(x_i,u_i)(\cdot)-(x,u)(\cdot)\|=0$. It means that  $(x_i,u_i)(\cdot)$ converge uniformly to $(x,u)(\cdot)$ on $[t_0,T]$. In particular there exists  $i_0$ such that $x_i([t_0,T])\subset\B_{R-1}$ for all $i\geq i_0$.

Now we show that  $(x_i,u_i)(t)\in \E U(t,\cdot)$ for all $t\in[t_0,T]$ and $i\geq i_0$. Let us fix  $i\geq i_0$. We define $\tau_{\bullet}:=\inf\{\,\tau\in[t_0,T]\,\mid\,\forall\,t\in[\tau,T]\;(x_i,u_i)(t)\in \E U(t,\cdot)\,\}$. Suppose, contrary to our claim, that $t_0<\tau_{\bullet}\leq T$. By lower semicontinuity of $U$ we have $(x_i,u_i)(\tau_{\bullet})\in \E U(\tau_{\bullet},\cdot)$. We define the function $f:\R\times\R^{\scriptscriptstyle N}\times\R\to\R\times\R^{\scriptscriptstyle N}\times\R$  by the formula
\begin{equation*}
f(t,x,u)=\left\{
\begin{array}{lcc}
(\,1,\,\bar{\e}(t_0,x,u,a_i(t_0)\,) & \tn{if} & t<t_0, \\[2mm]
(\,1,\,\bar{\e}(t,x,u,a_i(t)\,) & \tn{if} & t_0\leq t\leq T, \\[2mm]
(\,1,\,\bar{\e}(T,x,u,a_i(T)\,) & \tn{if} & t>T,
\end{array}
\right.
\end{equation*}
We observe that $f$ is continuous. Furthermore, in view of \eqref{prowblt-02}, we have  
\begin{equation}\label{prowblt-022}
\forall\;(t,x,u)\in\E U\cap (t_0,T+1)\times\B_R\times\R,\;\;-f(t,x,u)\in T_{\E U}(t,x,u).
\end{equation}
In view of Nagumo Backward Viability Local Theorem (\,see \cite{IPA}\,) there exist $t_{\bullet}\in(t_0,\tau_{\bullet})$ and a function $(s_{\bullet},x_{\bullet},u_{\bullet})(\cdot)$ of class $\mathcal{C}^1$ such that
\begin{displaymath}
\left\{
\begin{array}{l}
(\dot{s}_{\bullet}(t),\dot{x}_{\bullet}(t),\dot{u}_{\bullet}(t))=f(s_{\bullet}(t),x_{\bullet}(t),u_{\bullet}(t))\;\; \tn{for all}\;\; t\in[t_{\bullet},\tau_{\bullet}],\\[2mm]
(s_{\bullet},x_{\bullet},u_{\bullet})(t)\in\E U \;\;\tn{for all}\;\;t\in[t_{\bullet},\tau_{\bullet}],\\[2mm]
(s_{\bullet},x_{\bullet},u_{\bullet})(\tau_{\bullet})=(\tau_{\bullet},x_i(\tau_{\bullet}),u_i(\tau_{\bullet}))\in (t_0,T]\times\B_{R-1}\times\R.
\end{array}
\right.
\end{displaymath}
Observe that $s_{\bullet}(t)=t$ for all $t\in[t_{\bullet},\tau_{\bullet}]$. Therefore, $(x_{\bullet},u_{\bullet})(t)\in \E U(t,\cdot)$ for all $t\in[\tau_{\bullet},\tau_{\bullet}]$. Since $(x_{\bullet},u_{\bullet})(\cdot)$ is also a solution of  \eqref{pcpc} on  $[t_{\bullet},\tau_{\bullet}]$ with $(x_{\bullet},u_{\bullet})(\tau_{\bullet})=(x_i(\tau_{\bullet}),u_i(\tau_{\bullet}))$, we have $(x_{\bullet},u_{\bullet})(t)=(x_i,u_i)(t)$ for all $t\in[t_{\bullet},\tau_{\bullet}]$. Therefore $(x_i,u_i)(t)\in \E U(t,\cdot)$ for all $t\in[t_{\bullet},T]$, which contradicts the definition of $\tau_{\bullet}$.

Therefore  $(x_i,u_i)(t)\in \E U(t,\cdot)$ for all $t\in[t_0,T]$ and $i\geq i_0$. Passing to the limit as $i\to\infty$, we obtain $(x,u)(t)\in \E U(t,\cdot)$ for all $t\in[t_0,T]$.
\end{proof}

\subsection{Viability Theorem} We start by formulating the viability theorem:

\begin{Th}[Viability Theorem]\label{Twi5}
Assume that $L$ satisfies \tn{(L1)-(L5)}. Let $U$ be a proper and lower semicontinuous function satisfying the following condition:
\begin{equation}\label{swn}\begin{split}
& \tn{For every}\;(t,x)\in\D U\cap [0,T)\times\R^{\scriptscriptstyle N},\;\tn{every}\;(n^t,n^x,n^u)\in N_{\E U}(t,x,U(t,x)),\\[-1.3mm]
& \tn{there exist}\;(t_k,x_k)\to (t,x)\;\tn{and}\;\alpha_k\rightarrow 0,\;\tn{exists}\; v_k\in \D L(t_k,x_k,\cdot) \;\tn{such that} \\[-2mm]
&n^t+\langle v_k,n^x \rangle-n^uL(t_k,x_k,v_k)\leq \alpha_k\;\tn{for all}\;k\in\N.
\end{split}
\end{equation}
Then for every $(t_0,x_0)\in \D U\cap [0,T)\times\R^{\scriptscriptstyle N}$ there exists an absolutely continuous function $(x,u):[t_0,T]\rightarrow\R^{\scriptscriptstyle N}\times\R$ with $(x,u)(t_0)=(x_0,U(t_0,x_0))$ which satisfies $(\dot{x},\dot{u})(t)\in Q(t,x(t))$ for a.e. $t\in[t_0,T]$ and  $u(t)\geq U(t,x(t))$ for all $t\in[t_0,T]$.
\end{Th}

We notice that the boundary condition \eqref{swn} is new. In the next section we show that every lower semicontinuous solution of \eqref{rowhj} satisfies the new boundary condition \eqref{swn}; see Proposition~\ref{flw1}. However, not every lower semicontinuous solution of \eqref{rowhj} satisfies the following classic boundary condition:
\begin{equation}\label{swnc}\begin{split}
& \tn{For every}\;(t,x)\in\D U\cap [0,T)\times\R^{\scriptscriptstyle N},\;\tn{every}\;(n^t,n^x,n^u)\in N_{\E U}(t,x,U(t,x)),\\[-1.3mm]
& \tn{there exist}\; v\in \D L(t,x,\cdot) \;\tn{such that}\; n^t+\langle v,n^x \rangle-n^uL(t,x,v)\leq 0.
\end{split}\end{equation}

\noindent For instance, we consider the Lagrangian $L:[0,T]\times\R\times\R\to\overline{\R}$ given by 
\begin{equation}\label{expnun1}
L(t,x,v)=\left\{
\begin{array}{ccl}
+\infty & \tn{if} & |v|> 2,\;t\not=x, \\[1.5mm]
\frac{|v|}{2\sqrt{|t-x|}\exp\left(2\sqrt{|t-x|}\,\right)} & \tn{if} & |v|\leq 2,\;t\not=x, \\[2mm]
0 & \tn{if} & v=0,\; t=x,\\[0.5mm]
+\infty & \tn{if} & v\not=0,\;t=x.
\end{array}
\right.
\end{equation}
This Lagrangian satisfies (L1)-(L5); see \cite{AM0}.
Let $V:[0,T]\times\R\to\R$ be given by 
\begin{equation}\label{expnun2}
V(t,x)=\left\{
\begin{array}{ccl}
\exp\left(-2\sqrt{x-t}\:\right)-1 & \tn{if} & x\geq t, \\[1mm]
1-\exp\left(-2\sqrt{t-x}\:\right) & \tn{if} & 2t-T\leq x< t, \\[1mm]
1 & \tn{if} &  x<2t-T.
\end{array}
\right.
\end{equation}

\noindent Then $V$ is the value function associated with  $g(\cdot)=V(T,\cdot)$ and $L$; see \cite[Section 5]{AM0}. Moreover, $V$ is a lower semicontinuous solution of \eqref{rowhj}; see \cite[Thm. 4.2]{AM0}. We observe that $(1,-1,0)\in N_{\E V}(\xi,\xi,V(\xi,\xi))$, where $\xi\in(0,T)$. We suppose that \eqref{swnc} holds. Then  $1+(-1)\cdot 0-0\cdot L(\xi,\xi,0)\leq 0$, which is impossible. However, for $(t_k,x_k)=(\xi,\xi)-{\scriptstyle\frac{1}{2k}}(\xi,\xi\!-\!T)$, $\alpha_k={\scriptstyle\frac{1}{k}}$ and $v_k=1-{\scriptstyle\frac{1}{2k}}$, we have $1+(-1)\cdot v_k-0\cdot L(t_k,x_k,v_k)\leq \alpha_k$ for all $k\in\N$.

\vspace{2mm}
We notice that the above Lagrangian does not satisfy the condition (L6). Obviously, this condition is not required in Theorem \ref{Twi5}. Nevertheless, the natural question arises, whether  adding the condition (L6) helps. It turns out that it helps, but not much. Namely, if the Lagrangian $L$ satisfies (L1)-(L6), the set-valued map $t\to\D L(t,x,\cdot)$ is continuous in the Hausdorff sense for every $x\in\R^{\scriptscriptstyle N}$, and the set $\D L(t,x,\cdot)$ is closed for every $(t,x)\in[0,T]\times\R^{\scriptscriptstyle N}$, then the conditions \eqref{swn} and \eqref{swnc} are equivalent. In Example~\ref{ex-2} the~Lagrangian $L$ satisfies (L1)-(L6), but  $t\to\D L(t,\cdot)$ is not continuous. In Example~\ref{ex-1} the Lagrangian $L$ satisfies (L1)-(L6), but the set $\D L(x,\cdot)$ is not closed for all $x\in\R\setminus\{0\}$.

\vspace{2mm}
In the proof of Theorem \ref{Twi5} we construct a viable trajectory using Euler's broken lines and   methods from the monography of Cesari \cite{LC}, similarly to Plaskacz-Quincampoix in the proof of \cite[Theorem 3.19]{P-Q}.  There are two significant differences between proofs of Theorem 4.3 and \cite[Theorem 3.19]{P-Q}. The first one relates to the modification of the  $\varepsilon$-approximate solution. This change is needed, because in Theorem \ref{Twi5} we assume the new boundary condition \eqref{swn}, but Plaskacz-Quincampoix in \cite[Theorem 3.19]{P-Q} assume  the classic boundary condition. The modification of the proof can be done due to Lemma \ref{lem5},\linebreak which generalizes  \cite[Lemma 4.1]{P-Q}. The second difference relies on isolating a local version of  Theorem \ref{Twi5}. This decomposition is required, since in Theorem \ref{Twi5} functions $c(\cdot)$, $k_R(\cdot)$ might be unbounded, but  Plaskacz-Quincampoix in \cite[Theorem 3.19]{P-Q} work with constant functions $c(\cdot)$, $k_R(\cdot).$

\begin{Lem}\label{lem5}
If $\langle f,y-w \rangle<|y-w|^2$, then  the inequality
\begin{eqnarray}\label{lem5-1}
|y+h(f-(y-w))-w|\leq|y-w| 
\end{eqnarray} 
holds for each  $h$ satisfying 
\begin{eqnarray}\label{lem5-2}
0\leq h < 2\frac{|y-w|^2-\langle f,y-w \rangle}{|f-(y-w)|^2}.
\end{eqnarray} 
\end{Lem}

\begin{proof}
First, we show that \eqref{lem5-2} is well-defined. Indeed, if $f=y-w$, then    $$|y-w|^2=\langle y-w,y-w \rangle=\langle f,y-w \rangle<|y-w|^2,$$ which is impossible. Let $f\not=y-w$. Then, by  $\langle f,y-w \rangle<|y-w|^2$, we have
$$ 0 < 2\frac{|y-w|^2-\langle f,y-w \rangle}{|f-(y-w)|^2}. $$
Next, we show that \eqref{lem5-1} holds. By multiplying both sides of  \eqref{lem5-2} by $|f-(y-w)|^2$ and  moving the expressions from right to left we obtain
$$2\langle y-w,f \rangle-2|y-w|^2+h|f-(y-w)|^2<0.$$
We transform the above inequality as follows
\begin{eqnarray*}
0 &>& 2\langle y-w,f \rangle-2\langle y-w,y-w\rangle+h|f-(y-w)|^2 \\
  &\geq& 2\langle y-w,f-(y-w) \rangle + h|f-(y-w)|^2. 
\end{eqnarray*}
By multiplying both sides of the above inequality by $h$, we get
$$2h\langle y-w,f-(y-w) \rangle + h^2|f-(y-w)|^2\leq 0.$$
By adding $|y-w|^2$ to both sides of the above inequality, we have
$$|y-w|^2+2h\langle y-w,f-(y-w) \rangle + h^2|f-(y-w)|^2\leq |y-w|^2,$$
hence
$$|y-w+h(f-(y-w))|^2\leq|y-w|^2.$$
We observe that the above inequality implies \eqref{lem5-1}. 
\end{proof}

\begin{Th}[Local Viability Theorem]\label{Twi5-l}
Assume that $L$ satisfies \tn{(L1)-(L4)}. Let $U$ be a proper and lower semicontinuous function satisfying the following condition:
\begin{equation}\label{swnl}\begin{split}
& \tn{For every}\;(t,x,u)\in\E U\cap [0,T)\times\R^{\scriptscriptstyle N}\times\R,\;\tn{every}\;(n^t,n^x,n^u)\in N_{\E U}(t,x,u),\\[-1.3mm]
& \tn{there exist}\;(t_k,x_k)\to (t,x)\;\tn{and}\;\alpha_k\rightarrow 0,\;\tn{exists}\; v_k\in \D L(t_k,x_k,\cdot) \;\tn{such that} \\[-2mm]
&n^t+\langle v_k,n^x \rangle-n^uL(t_k,x_k,v_k)\leq \alpha_k\;\tn{for all}\;k\in\N.
\end{split}\end{equation}
Then for every  $(t_0,x_0,u_0)\in \E U\cap [0,T)\times\R^{\scriptscriptstyle N}\times\R$ there exist $T_0\in(t_0,T)$ and an absolutely continuous function $(x,u):[t_0,T_0]\rightarrow\R^{\scriptscriptstyle N}\times\R$ with $(x,u)(t_0)=(x_0,u_0)$ which satisfies $(\dot{x},\dot{u})(t)\in Q(t,x(t))$ for a.e. $t\in[t_0,T_0]$ and $u(t)\geq U(t,x(t))$ for all $t\in[t_0,T_0]$.
\end{Th}

\begin{Rem}
In Theorem \ref{Twi5-l} the condition (L5) is not required. Moreover, the conditions\linebreak \eqref{swn} and \eqref{swnl} are equivalent, since $N_{\E U}(t,x,u)\subset N_{\E U}(t,x,U(t,x))$ for   $(t,x,u)\in\E U$.
\end{Rem}

\begin{proof}[Proof Theorem \ref{Twi5-l}]
Fix  $(t_0,x_0,u_0)\in \E U\cap [0,T)\times\R^{\scriptscriptstyle N}\times\R$ and set $R:=|x_0|+2$. Define $L_{\bullet}(t,x,u,v):=L(\pi_{\scriptscriptstyle[0,T]}(t),\pi_{\scriptscriptstyle\B_R}(x),v)$ and $Q_{\bullet}(t,x,u):=Q(\pi_{\scriptscriptstyle[0,T]}(t),\pi_{\scriptscriptstyle\B_R}(x))$ for every $t,u\in\R$, $x,v\in\R^{\scriptscriptstyle N}$. In view of (L4) there exists a constant $C_R\geq 0$ such that  $\|\D L(t,x,\cdot)\|\leq C_R$ for every $(t,x)\in[0,T]\times\B_R$. Therefore $\|\D L_{\bullet}(t,x,u,\cdot)\|\leq C_R$ for  every $t,u\in\R$, $x\in\R^{\scriptscriptstyle N}$. Since functions $U$ and $L$ are proper and lower semicontinuous, there exists a constant $D>0$ such that $U(t,x)\geq -D$ and $L(t,x,v)\geq -D$ for every  $t\in[0,T]$, $x\in\B_R$, $v\in\B_{C_R}$. Therefore $L_{\bullet}(t,x,u,v)\geq-D$ for  every $t,u\in\R$, $x,v\in\R^{\scriptscriptstyle N}$.

\bf{\textsc{Step 1.} Definition and properties of an $\pmb{\varepsilon}$-approximate solution.} 
Let $\varepsilon$ and $t_\varepsilon$ satisfy
\begin{align}
&0<\varepsilon\leq\varepsilon_0:=(T-t_0)\,[2(1+T)]^{-1},\label{ne}\\[1mm]
&t_0<t_\varepsilon\leq T_0:=\min\left\{t_0+[1+C_R]^{-1},t_0+(T-t_0)\,2^{-1}\right\}.\label{nee}
\end{align}
We say that a family $\sum=\{[t_j,\tau_j)\mid j\in J\}$ of nonempty intervals is a
subdivision of the interval  $[t_0,t_\varepsilon)$ if $[t_j,\tau_j)\cap[t_i,\tau_i)=\emptyset$ for all $j\not=i$ and $[t_0,t_\varepsilon)=\bigcup_{j\in J}[t_j,\tau_j)$.
Let  $y(\cdot)=(s(\cdot),x(\cdot),u(\cdot))$ be an absolutely continuous function defined on $[t_0,t_\varepsilon)$ into $\R\times\R^{\scriptscriptstyle N}\times\R$ such that  $y(t_0)=(t_0,x_0,u_0)$. We say that a triple $\big([t_0,t_\varepsilon),y(\cdot),\sum\big)$ is an $\varepsilon$-approximate\linebreak solution if the following inequality holds 
\begin{equation}\label{neea}
\dist(y(t),\E U)\leq \varepsilon\;\;\tn{for all}\;\, t\in[t_0,t_\varepsilon),
\end{equation}
and for all $j\in J$ there exist $f_j$, $w_j$, $\bar{w}_j$ such that
\begin{enumerate}
\item[\bf{(i)}] $\forall$ $t\in[t_j,\tau_j)$\;\; $y(t)=y(t_j)+(t-t_j)(f_j-(y(t_j)-w_j))$,
\item[\bf{(ii)}] $f_j\in\{1\}\times Q_{\bullet}(\bar{w}_j)$,\;\; $|\bar{w}_j-w_j|\leq \varepsilon$,\;\;$w_j\in \E U$,
\item[\bf{(iii)}] $|w_j-y(t_j)|=\dist(y(t_j),\E U)$,
\item[\bf{(iv)}] If $y(t_j)\in \E U$, then $(\tau_j-t_j)|f_j|\leq\varepsilon$,
\item[\bf{(v)}] If $y(t_j)\notin \E U$, then $|y(t)-w_j|\leq|y(t_j)-w_j|$ for all $t\in[t_j,\tau_j)$.
\end{enumerate}
We show that an $\varepsilon$-approximate solution satisfies the following inequality  
\begin{equation}\label{n5}
|w_j-y(t)|\leq \varepsilon\;\;\;\tn{for all}\;\;t\in[t_j,\tau_j).
\end{equation}
Indeed, suppose that $y(t_j)\in \E U$. Then by (iii), (i) we get $w_j=y(t_j)$, $|w_j-y(t)|=(t-t_j)|f_j|$.  The latter, together with (iv), implies (\ref{n5}). Now, suppose that  $y(t_j)\notin \E U$. Then by (iii), \eqref{neea} we get $|w_j-y(t_j)|\leq\varepsilon$. The latter, together with (v), implies (\ref{n5}).

We need the following notations:
\begin{align*}
&(t_j^w,x_j^w,u_j^w)=w_j\,,\;\; (n_j^t,n_j^x,n_j^u)=y(t_j)-w_j\,,\\
&(1,v_j,-\eta_j)=f_j\,,\; \tn{where}\; \eta_j\geq L_{\bullet}(\bar{w}_j,v_j)\,\; \tn{and}\;|\bar{w}_j-w_j|\leq \varepsilon.
\end{align*}
We also need auxiliary functions defined by
\begin{equation*}
\begin{array}{lll}
n^u=\sum\limits_{j\in J}\chi_{[t_j,\tau_j)}n_j^u\,, & n^x=\sum\limits_{j\in J}\chi_{[t_j,\tau_j)}n_j^x\,, & n^t=\sum\limits_{j\in J}\chi_{[t_j,\tau_j)}n_j^t\,,\\
v=\sum\limits_{j\in J}\chi_{[t_j,\tau_j)}v_j\,, & \eta=\sum\limits_{j\in J}\chi_{[t_j,\tau_j)}\eta_j\,.
\end{array}
\end{equation*}
Note that in view of (i) and  (\ref{n5}) we have 
\begin{align}
&\dot{s}(t)=1-n^t(t)\;\; \tn{for a.e.}\;\, t\in[t_0,t_\varepsilon)\,, \hspace{12.5mm} |n^t(t)|\leq\varepsilon\;\;\tn{for all}\;\,t\in[t_0,t_\varepsilon)\,,\label{npo2}\\
&\dot{x}(t)=v(t)-n^x(t)\;\; \tn{for a.e.}\;\, t\in[t_0,t_\varepsilon)\,, \hspace{7.5mm} |n^x(t)|\leq\varepsilon\;\;\tn{for all}\;\,t\in[t_0,t_\varepsilon)\,,\label{npo}\\
&\dot{u}(t)=-\eta(t)-n^u(t)\;\;\tn{for a.e.}\;\, t\in[t_0,t_\varepsilon)\,,\;\;\;\;\;|n^u(t)|\leq\varepsilon\;\;\tn{for all}\;\,t\in[t_0,t_\varepsilon)\,.\label{npo1}
\end{align}

We show that the following properties hold: 	
\begin{equation}\label{ofax}
\begin{array}{lll}
|s(t)-t|\leq T_0\,\varepsilon,\;\;|x(t)|\leq R-1 &\tn{for all}& t\in[t_0,t_\varepsilon),\\
|\dot{s}(t)|\leq 2\;\;\,\tn{and}\;\;\,|\dot{x}(t)|\leq C_R+1 &\tn{for a.e.}& t\in[t_0,t_\varepsilon).
\end{array}
\end{equation}
Indeed, by \eqref{ne} and \eqref{npo2}  we get $|\dot{s}(t)|\leq 1+|n^t(t)|\leq 2$. Moreover, $s(\cdot)$ is absolutely continuous, so $|s(t)-t|=|t_0+\int_{t_0}^t\dot{s}(\varsigma)\,d\varsigma-t|=|t_0+\int_{t_0}^t(1-n^t(\varsigma))\,d\varsigma-t| =|\int_{t_0}^tn^t(\varsigma)\,d\varsigma|\leq T_0\,\varepsilon$. Since $v_j\in\D L_{\bullet}(\bar{w}_j,\cdot)$, we get $|v_j|\leq C_R$. The latter, together with \eqref{ne} and \eqref{npo}, implies $|\dot{x}(t)|\leq|v(t)|+|n^x(t)|\leq C_R+1$. Hence, using absolute continuity of $x(\cdot)$, we have $|x(t)|\leq|x_0|+\int_{t_0}^t|\dot{x}(\varsigma)|\,d\varsigma\leq|x_0|+(t_\varepsilon-t_0)(C_R+1)$. Therefore, in view of \eqref{nee}, we obtain $|x(t)|\leq|x_0|+(T_0-t_0)(C_R+1)\leq|x_0|+1=R-1$.

We also show that the following properties hold:
\begin{equation}\label{2nn}
\begin{array}{lll}
|u(t)|\leq |u_0|+(T_0+1)(D+1) &\tn{for all}& t\in[t_0,t_\varepsilon),\\
\dot{u}(t)\leq D+1 \;\;\tn{and}\;-D\leq\eta(t) &\tn{for a.e.}& t\in[t_0,t_\varepsilon).
\end{array}
\end{equation}
Indeed, by $\eta_j\geq L_{\bullet}(\bar{w}_j,v_j)\geq -D$  we get $\eta(t)\geq -D$.  The latter, together with \eqref{npo1}, implies $\dot{u}(t)=-\eta(t)-n^u(t)\leq D+1$. Hence, using absolute continuity of $u(\cdot)$, we obtain $u(t)=u_0+\int_{t_0}^t\dot{u}(\varsigma)\,d\varsigma\leq |u_0|+T_0(D+1)$. The inequality $|u(t)|\leq |u_0|+(T_0+1)(D+1)$ will be proved once we prove that $-D-1\leq u(t)$.
Let $w_t=(s_t,x_t,u_t)\in \E U$ satisfy $|y(t)-w_t|=\dist(y(t),\E U)$. Then, using \eqref{neea}, we obtain $|y(t)-w_t|\leq \varepsilon\leq 1$. Therefore $|x(t)-x_t|\leq 1$ and $|u(t)-u_t|\leq 1$. Hence, using \eqref{ofax}, we get $|x_t|\leq |x(t)|+1\leq R$. Moreover, because of  $w_t=(s_t,x_t,u_t)\in \E U$, we have $s_t\in[0,T]$ and $U(s_t,x_t)\leq u_t$. Since $(s_t,x_t)\in[0,T]\times\B_R$, we obtain $-D\leq U(s_t,x_t)$. Therefore $-D\leq U(s_t,x_t)\leq u_t\leq u(t)+1$. So $-D-1\leq u(t)$.

\bf{\textsc{Step 2.} Extension of an $\pmb{\varepsilon}$-approximate solution.} We first show that if a triple $\big([t_0,t_\varepsilon),y(\cdot),\sum\big)$ is an $\varepsilon$-approximate solution, then the limit $\lim_{t\to t_\varepsilon}y(t)$ exists and $y(\cdot)$ is an absolutely continuous function  on $[t_0,t_\varepsilon]$, where $y(t_\varepsilon):=\lim_{t\to t_\varepsilon}y(t)$. Hence, by \eqref{neea},
\begin{equation}\label{neealg}
\dist(y(t_\varepsilon),\E U)=\lim\nolimits_{t\to t_\varepsilon}\dist(y(t),\E U)\leq\varepsilon.
\end{equation}
Since $(s,x)(\cdot)$ is absolutely continuous on $[t_0,t_\varepsilon)$ and  $(\dot{s},\dot{x})(\cdot)$ is bounded almost everywhere on $[t_0,t_\varepsilon)$ (from \eqref{ofax}), it follows that the limit  $\lim_{t\to t_\varepsilon}(s,x)(t)$ exists and $(s,x)(\cdot)$ is absolutely continuous  on $[t_0,t_\varepsilon]$, where $(s,x)(t_\varepsilon):=\lim_{t\to t_\varepsilon}(s,x)(t)$. Let $t_n:=t_\varepsilon-{\scriptstyle\frac{1}{2n}}(t_\varepsilon-t_0)$ for all $n\in\N$. In view of \eqref{2nn} we obtain
\begin{eqnarray}\label{var-u-1}
\int_{t_0}^{t_\varepsilon}\chi_{[t_0,t_n]}(t)\,|\dot{u}(t)|\,dt &=& \int_{t_0}^{t_n}2\max\{\dot{u}(t),0\}\,dt - \int_{t_0}^{t_n}\dot{u}(t)\,dt\nonumber\\
&\leq & 2(t_n-t_0)(D+1)-u(t_n)+u(t_0)\nonumber\\
&\leq & 2|u_0|+(3T_0+1)(D+1).
\end{eqnarray}
Due to the above inequality and Lebesgue's Monotone Convergence Theorem we obtain that $\dot{u}(\cdot)$ is integrable on $[t_0,t_\varepsilon]$. Hence it follows that the limit $\lim_{t\to t_\varepsilon}u(t)$ exists and $u(\cdot)$ is absolutely continuous  on $[t_0,t_\varepsilon]$, where $u(t_\varepsilon):=\lim_{t\to t_\varepsilon}u(t)$.

Now, we show that if  $\big([t_0,t_\varepsilon),y(\cdot),\sum\big)$ is an $\varepsilon$-approximate solution with $t_\varepsilon<T_0$, then there exists an $\varepsilon$-approximate solution $\big([t_0,\tau_\varepsilon),\bar{y}(\cdot),\bar{\sum}\big)$ with $t_\varepsilon<\tau_\varepsilon<T_0$ which satisfies $y(t)=\bar{y}(t)$ for all $t\in[t_0,t_\varepsilon)$ and $\sum\subset\bar{\sum}$. Indeed, suppose that $y(t_\varepsilon)\in \E U$. Then we take an arbitrary $f_\varepsilon\in\{1\}\times Q_{\bullet}(y(t_\varepsilon))$ and $\tau_\varepsilon\in(t_\varepsilon,T_0)$ such that $(\tau_\varepsilon-t_\varepsilon)|f_\varepsilon|\leq\varepsilon$. Let $w_\varepsilon=y(t_\varepsilon)=\bar{w}_\varepsilon$.\linebreak We define the function $\bar{y}(\cdot)$ on $[t_0,\tau_\varepsilon)$ by the formula $\bar{y}(t):=y(t)$ for every $t\in[t_0,t_\varepsilon)$ and $\bar{y}(t):=y(t_\varepsilon)+(t-t_\varepsilon)(f_\varepsilon-(y(t_\varepsilon)-w_\varepsilon))$ for every $t\in[t_\varepsilon,\tau_\varepsilon)$. We define the set $\bar{\sum}$  by $\sum\cup\big\{[t_\varepsilon,\tau_\varepsilon)\big\}$. 
We notice that $\big([t_0,\tau_\varepsilon),\bar{y}(\cdot),\bar{\sum}\big)$ defined as above satisfies \eqref{neea} and (i)-(v). In particular, due to $y(t_0)\in \E U$, the family of all $\varepsilon$-approximate solutions is nonempty. Suppose that $y(t_\varepsilon)\notin \E U$. Then, due to \eqref{neealg}, we have $0<\dist(y(t_\varepsilon),\E U)\leq\varepsilon$. Let $w_\varepsilon=(s_\varepsilon,x_\varepsilon,u_\varepsilon)\in \E U$ be a proximal point of $y(t_\varepsilon)$ in $\E U$, then $w_\varepsilon\not=y(t_\varepsilon)$ and
\begin{equation}\label{dis}
|y(t_\varepsilon)-w_\varepsilon|=\dist(y(t_\varepsilon),\E U)\leq\varepsilon.
\end{equation}
Due to the properties of the normal cone from Subsection \ref{nap} we have 
\begin{equation}\label{wpgh}
y(t_\varepsilon)-w_\varepsilon\in N_{\E U}(w_\varepsilon).
\end{equation}
Since $(s_\varepsilon,x_\varepsilon,u_\varepsilon)\in \E U$, we get  $s_\varepsilon\in[0,T]$. We show that $s_\varepsilon<T$. Indeed, by \eqref{dis}, \eqref{ofax} we get $|s_\varepsilon-t_\varepsilon|\leq|s_\varepsilon-s(t_\varepsilon)|+|s(t_\varepsilon)-t_\varepsilon|\leq\varepsilon+T_0\,\varepsilon$.
Hence, using \eqref{ne}, \eqref{nee}, we obtain
\begin{equation*}
s_\varepsilon\,\leq\, t_\varepsilon+\varepsilon+T_0\,\varepsilon\,<\,t_\varepsilon+\varepsilon\,(1+T)\,\leq\, t_0+{\textstyle\frac{1}{2}}(T-t_0)+{\textstyle\frac{1}{2}}(T-t_0)=T.
\end{equation*}
Therefore $(s_\varepsilon,x_\varepsilon,u_\varepsilon)\in \E U\cap[0,T)\times\R^{\scriptscriptstyle N}\times\R$. Moreover, we know that \eqref{wpgh}  holds true. 
In view of \eqref{swnl} there exist $(t_k,x_k)\to(s_\varepsilon,x_\varepsilon)$, $\alpha_k\rightarrow 0$ and $v_k\in \D L(t_k,x_k,\cdot)$ such that 
\begin{equation}\label{lneq}
\langle (1,v_k,-L(t_k,x_k,v_k)), y(t_\varepsilon)-w_\varepsilon\rangle\leq \alpha_k\;\;\tn{for all}\;\;k\in\N.
\end{equation}
Let $f_k:=(1,v_k,-L(t_k,x_k,v_k))$ and $\bar{w}_k:=(t_k,x_k,u_\varepsilon)$ for all $k\in\N$. One can choose $k_0\in\N$ such that $\alpha_{k_0}<|y(t_\varepsilon)-w_\varepsilon |^2$ and $|\bar{w}_{k_0}-w_\varepsilon|\leq\varepsilon$. Therefore $\langle f_{k_0},y(t_\varepsilon)-w_\varepsilon\rangle<|y(t_\varepsilon)-w_\varepsilon |^2$. In view of Lemma \ref{lem5} we can choose $\tau_\varepsilon$ such that
$t_\varepsilon<\tau_\varepsilon<T_0$ and
\begin{equation}\label{ccld}
|y(t_\varepsilon)+(t-t_\varepsilon)(f_{k_0}-(y(t_\varepsilon)-w_\varepsilon))-w_\varepsilon|\leq |y(t_\varepsilon)-w_\varepsilon|\;\;\tn{for all}\;\;t\in[t_\varepsilon,\tau_\varepsilon).
\end{equation}
Set $\bar{w}_\varepsilon:=\bar{w}_{k_0}$ and $f_\varepsilon:=f_{k_0}$. We define the function $\bar{y}(\cdot)$ on $[t_0,\tau_\varepsilon)$ by the formula $\bar{y}(t):=y(t)$ for all $t\in[t_0,t_\varepsilon)$ and $\bar{y}(t):=y(t_\varepsilon)+(t-t_\varepsilon)(f_\varepsilon-(y(t_\varepsilon)-w_\varepsilon))$ for all $t\in[t_\varepsilon,\tau_\varepsilon)$. We define the set $\bar{\sum}$  by $\sum\cup\big\{[t_\varepsilon,\tau_\varepsilon)\big\}$. By \eqref{ccld} we get $|\bar{y}(t)-w_\varepsilon|\leq |y(t_\varepsilon)-w_\varepsilon|$ for all $t\in[t_\varepsilon,\tau_\varepsilon)$.\linebreak
The latter, together with \eqref{dis}, implies that $\dist(\bar{y}(t),\E U) \leq \varepsilon$ for all $t\in[t_\varepsilon,\tau_\varepsilon)$. Since $|\bar{w}_{k_0}-w_\varepsilon|\leq\varepsilon$, we have $|x_{k_0}-x_\varepsilon|\leq\varepsilon$. In view of \eqref{dis} we get $|x(t_\varepsilon)-x_\varepsilon|\leq\varepsilon$. Therefore $|x_{k_0}-x(t_\varepsilon)|\leq 2\varepsilon$. The latter, together with \eqref{ofax}, implies $$|x_{k_0}|\leq|x(t_\varepsilon)|+2\varepsilon\leq R-1+2\varepsilon\leq R-1+1=R.$$
Since ($t_{k_0},x_{k_0})\in[0,T]\times\B_R$, we have
\begin{equation*}
f_\varepsilon=f_{k_0}\in\{1\}\times Q(t_{k_0},x_{k_0})=\{1\}\times Q_{\bullet}(\bar{w}_{k_0})=\{1\}\times Q_{\bullet}(\bar{w}_\varepsilon).
\end{equation*}
Therefore the triple $\big([t_0,\tau_\varepsilon),\bar{y}(\cdot),\bar{\sum}\big)$ defined above satisfies \eqref{neea} and (i)-(v).

In the family of all $\varepsilon$-approximate solutions we define a partial order relation\linebreak $\big([t_0,t_\varepsilon),y(\cdot),\sum\big)\!\preccurlyeq\!\big([t_0,\tau_\varepsilon),\bar{y}(\cdot),\bar{\sum}\big)$ if and only if $[t_0,t_\varepsilon)\!\subset\![t_0,\tau_\varepsilon)$, $\sum\subset\bar{\sum}$, and $y(t)\!=\!\bar{y}(t)$ for all $[t_0,t_\varepsilon)$. We observe that in the family of all $\varepsilon$-approximate solutions every chain   $\left\{\!\big([t_0,t_{\varepsilon}^{\scriptscriptstyle\lambda}),y_\lambda(\cdot),\sum_\lambda\big)\!\right\}_{\!\!\lambda\in\Lambda}$  has an upper bound $\big([t_0,\tau_\varepsilon),\bar{y}(\cdot),\bar{\sum}\big)$, where $[t_0,\tau_\varepsilon):=\bigcup_{\lambda\in\Lambda}[t_0,t_{\varepsilon}^{\scriptscriptstyle\lambda})$, $\bar{\sum}:=\bigcup_{\lambda\in\Lambda}\bar{\sum}$, and $\bar{y}(t):=y_\lambda(t)$ for all $t\in[t_0,t_{\varepsilon}^{\scriptscriptstyle\lambda})$, $\lambda\in\Lambda$. By Kuratowski-Zorn Lemma there exists a maximal element $\big([t_0,t_\varepsilon^*),y_*(\cdot),\sum_*\big)$. We show that $t_\varepsilon^*=T_0$. Suppose, contrary to our claim, that $t_\varepsilon^*<T_0$. Then there exists  an $\varepsilon$-approximate solution $\big([t_0,\tau_\varepsilon^*),\bar{y}_*(\cdot),\bar{\sum}_*\big)$ such that $\big([t_0,t_\varepsilon^*),y_*(\cdot),\sum_*\big)\prec\big([t_0,\tau_\varepsilon^*),\bar{y}_*(\cdot),\bar{\sum}_*\big)$, which contradicts the definition of the maximal element. So there is an $\varepsilon$-approximate solution on the whole interval $[t_0,T_0]$.

\bf{\textsc{Step 3.} Convergence of approximate solutions.} Let $y_n(\cdot)=(s_n(\cdot),x_n(\cdot),u_n(\cdot))$ be an $\frac{1}{n}$-approximate solution defined on $[t_0,T_0]$ for all $n\in\N_0:=\N\cap[1/\varepsilon_0,\infty)$. 

The function $(s_n,x_n)(\cdot)$ is absolutely continuous on $[t_0,T_0]$ and $(s_n,x_n)(t_0)=(t_0,x_0)$ for every $n\in\N_0$. In view of \eqref{ofax} the family $\{(s_n,x_n)(\cdot)\}_{n\in\N_0}$ is equi-bounded and the family $\{(\dot{s}_n,\dot{x}_n)(\cdot)\}_{n\in\N_0}$ is equi-absolutely integrable. Therefore, in view of Arzel\`a-Ascoli and Dunford-Pettis Theorems, there exists a subsequence (we do not relabel) such that $(s_n,x_n)(\cdot)$ converges uniformely to an absolutely continuous function $(s,x):[t_0,T_0]\to\R^{\scriptscriptstyle 1+N}$  and $(\dot{s}_n,\dot{x}_n)(\cdot)$ converges weakly in $L^1([t_0,T_0],\R^{\scriptscriptstyle 1+N})$ to $(\dot{s},\dot{x})(\cdot)$. Moreover, by  \eqref{ofax}, we get $|s(t)-t|=\lim_{n\to\infty}|s_n(t)-t|\leq\lim_{n\to\infty}\frac{1}{n}T_0=0$ and $|x(t)|=\lim_{n\to\infty}|x_n(t)|\leq R-1$. Therefore $s(t)=t$ and $x(t)\in\B_R$ for all $t\in[t_0,T_0]$. Additionally, $(s,x)(t_0)=(s_0,x_0)$.

The function $u_n(\cdot)$ is absolutely continuous on $[t_0,T_0]$ and $u_n(t_0)=u_0$ for every $n\in\N_0$. In view of \eqref{2nn} the family $\{u_n(\cdot)\}_{n\in\N_0}$ is equi-bounded and for all $n\in\N_0$ we get
\begin{eqnarray*}
\mathrm{Var}_{[t_0,T_0]}u_n(\cdot) &\leq &\int_{t_0}^{T_0}|\dot{u}_n(t)|\,dt \;\;=\;\; \int_{t_0}^{T_0}2\max\{\dot{u}_n(t),0\}\,dt - \int_{t_0}^{T_0}\dot{u}_n(t)\,dt\\
&\leq & 2(T_0-t_0)(D+1)-u_n(T_0)+u_n(t_0)\;\;\leq\;\;  2|u_0|+(3T_0+1)(D+1).
\end{eqnarray*}
The above inequality implies that the variations of functions $u_n(\cdot)$ on $[t_0,T_0]$ are equi-bounded. Therefore, in view of Helly Theorem (cf. Theorem 15.1.i in \cite{LC}), there exists\linebreak a subsequence (we do not relabel) such that $u_n(\cdot)$ converges pointwise (everywhere) to a bounded variation function $\tilde{u}:[t_0,T_0]\to\R$. Therefore, $\tilde{u}(t_0)=\lim_{n\to\infty}u_n(t_0)=u_0$.\linebreak  Obviously, the function $\tilde{u}(\cdot)$ may not be absolutely continuous.

Since $\dot{x}_{\sigma+k}(\cdot)\rightharpoonup\dot{x}(\cdot)$ as $k\to\infty$ in $L^1([t_0,T_0],\R^{\scriptscriptstyle N})$ for all $\sigma\in\N_0$, by the Mazur Theorem, there exist real numbers $\lambda_{k,N}^{\sigma}\geq 0$ for  $k=1,2,\ldots,N$ and $N\in\N$  such that $\sum_{k=1}^{N}\lambda_{k,N}^{\sigma}=1$ and $\sum_{k=1}^{N}\lambda_{k,N}^{\sigma}\dot{x}_{\sigma+k}(\cdot)\longrightarrow\dot{x}(\cdot)$ as $N\to\infty$ in $L^1([t_0,T_0],\R^{\scriptscriptstyle N})$ for all $\sigma\in\N_0$. Then for all $\sigma\in\N_0$ there exists an increasing sequence  $\{N_{m}^\sigma\}_{m\in\N}$ such that 
\begin{equation*}
z_{m}^{\sigma}(t):=\sum_{k=1}^{N_{m}^\sigma}\lambda_{k,N_{m}^\sigma}^{\sigma}\dot{x}_{\sigma+k}(t)\longrightarrow\dot{x}(t)\;\;\;\tn{as}\;\;\;m\to\infty\;\;\;\;
\tn{a.e.}\;\;\;t\in[t_{0},T_0]. \label{g1}
\end{equation*} 
For a.e. $t\in[t_0,T_0]$  we set
$$\eta_{m}^{\sigma}(t):=\sum_{k=1}^{N_{m}^\sigma}\lambda_{k,N_m^\sigma}^{s}\eta_{\sigma+k}(t),$$
by \eqref{2nn} we get $\eta_{m}^{\sigma}(t)\geq -D$ for a.e. $t\in[t_0,T_0]$ and all $\sigma\in\N_0$, $m\in\N$, and
$$\eta^{\sigma}(t):=\liminf\nolimits_{m\to\infty}\eta_{m}^{\sigma}(t),\qquad\eta(t):=\liminf\nolimits_{\sigma\to\infty}\eta^{\sigma}(t),$$
we observe that $\eta^{\sigma}(t)\geq -D$ and $\eta(t)\geq -D$  for a.e. $t\in[t_0,T_0]$ and all $\sigma\in\N_0$.

Now we show that for all $\tau_0\in[t_0,T_0]$ the following inequality is true:
\begin{equation}\label{nc11}
\int_{t_0}^{\tau_0}\eta(t)\:dt\leq \tilde{u}(t_0)-\tilde{u}(\tau_0).
\end{equation}
Indeed, fix $\tau_0\in(t_0,T_0]$, $\varepsilon>0$. We can choose $\sigma_0\in\N_0$ such that 
$\tilde{u}(\tau_0)-\varepsilon\leq u_{\sigma+k}(\tau_0)$ and $(\tau_0-t_0)/(\sigma+k)\leq\varepsilon$ for all $\sigma>\sigma_0$, $k\in\N$. In view of \eqref{npo1} we get $\dot{u}_n(t)=-\eta_n(t)-n_n^u(t)$, $|n_n^u(t)|\leq \frac{1}{n}$ for a.e. $t\in[t_0,T_0]$ and all $n\in\N_0$.
 Then for all $m\in\N$ and  $\sigma>\sigma_0$ we have
\begin{eqnarray*}
\int_{t_0}^{\tau_0}\eta_m^\sigma(t)\,dt & = & \sum_{k=1}^{N_m^\sigma}\lambda_{k,N_m^\sigma}^\sigma\int_{t_0}^{\tau_0}\eta_{\sigma+k}(t)\,dt\;=\;\sum_{k=1}^{N_m^\sigma}\lambda_{k,N_m^\sigma}^\sigma\left(\!\!-\!\int_{t_0}^{\tau_0}\dot{u}_{\sigma+k}(t)\,dt-\!\int_{t_0}^{\tau_0}n_{\sigma+k}^u(t)\,dt\right)\\
& \leq & \sum_{k=1}^{N_m^\sigma}\lambda_{k,N_m^\sigma}^\sigma\left(u_{\sigma+k}(t_0)\!-\!u_{\sigma+k}(\tau_0)\!+\!\frac{\tau_0-t_0}{\sigma+k}\right)\,\leq\,\sum_{k=1}^{N_m^\sigma}\lambda_{k,N_m^\sigma}^\sigma\Big(\tilde{u}(t_0)\!-\!\tilde{u}(\tau_0)\!+\!2\varepsilon\Big)\\
& = & \tilde{u}(t_0)-\tilde{u}(\tau_0)+2\varepsilon.
\end{eqnarray*}
By the Fatou Lemma we obtain
\begin{equation*}
\int_{t_0}^{\tau_0}\eta(t)\,dt \;\leq\; \liminf_{\sigma\to\infty} \int_{t_0}^{\tau_0}\eta^\sigma(t)\,dt\;\leq\; \liminf_{\sigma\to\infty} \liminf_{m\to\infty}\int_{t_0}^{\tau_0}\eta_m^\sigma(t)\,dt
\;\leq\;\tilde{u}(t_0)-\tilde{u}(\tau_0)+2\varepsilon.
\end{equation*}
Since   $\varepsilon>0$ is arbitrary, we conclude that the inequality \eqref{nc11} is true.

Now we show that for a.e. $t\in[t_0,T_0]$ we have 
\begin{equation}\label{ink11}
(\dot{x}(t),-\eta(t))\in Q_{\bullet}(t,x(t),\tilde{u}(t)).
\end{equation}
Indeed, fix $t\in(t_0,T_0)$ and $\varepsilon>0$. In view of \eqref{n5} and (ii)  we have 
$|w_{\!j_n}^n-y_n(t)|\leq\frac{1}{n}$ and $|\bar{w}_{\!\!j_n}^n-w_{\!j_n}^n|\leq\frac{1}{n}$. Therefore $|\bar{w}_{\!\!j_n}^n-y_n(t)|\leq \frac{2}{n}$ for all $n\in\N_0$. For all large $n\in\N_0$ we obtain 
$|y_n(t)-(t,x(t),\tilde{u}(t))|\leq\varepsilon$ and $|\bar{w}_{\!\!j_n}^n-y_n(t)|\leq\varepsilon$. Therefore $|\bar{w}_{\!\!j_n}^n-(t,x(t),\tilde{u}(t))|\leq\varepsilon$ for all large $n\in\N_0$. In view of (ii) we get $(v_n(t),-\eta_n(t))=(v_{\!\!j_n}^n,-\eta_{\!\!j_n}^n)\in Q_{\bullet}(\bar{w}_{\!\!j_n}^n)$. Hence, it follows that 
for all large $n\in\N_0$ we have
$$(v_{n}(t),-\eta_{n}(t))\in Q_{\bullet}(t,x(t),\tilde{u}(t);\varepsilon).$$
In view of \eqref{npo} we get $\dot{x}_n(t)\in\left(v_n(t)+\frac{1}{n}\B\right)$, so for all large $n\in\N_0$ we get
$$(\dot{x}_{n}(t),-\eta_{n}(t))\in Q_{\bullet}(t,x(t),\tilde{u}(t);\varepsilon)+{\textstyle\frac{1}{n}}\big(\B\times[-1,1]\big).$$
Hence for all large $\sigma\in\N_0$  and all $m\in\N$ we have
$$(z_m^\sigma(t),-\eta_m^\sigma(t))\in \mathrm{conv}\,Q_{\bullet}(t,x(t),\tilde{u}(t);\varepsilon)+{\textstyle\frac{1}{\sigma}}\big(\B\times[-1,1]\big).$$
Passing to a subsequence as $m\to\infty$ we have
$$(\dot{x}(t),-\eta^\sigma(t))\in \mathrm{cl}\,\mathrm{conv}\,Q_{\bullet}(t,x(t),\tilde{u}(t);\varepsilon)+{\textstyle\frac{1}{\sigma}}\big(\B\times[-1,1]\big),$$
passing to a subsequence as $\sigma\rightarrow\infty$ we have
$$(\dot{x}(t),-\eta(t))\in \mathrm{cl}\,\mathrm{conv}\,Q_{\bullet}(t,x(t),\tilde{u}(t);\varepsilon).$$
Since   $\varepsilon>0$ is arbitrary, we obtain
$$(\dot{x}(t),-\eta(t))\in \bigcap_{\varepsilon>0}\mathrm{cl}\,\mathrm{conv}\, 
Q_{\bullet}(t,x(t),\tilde{u}(t);\varepsilon).$$
The latter together with Lemma \ref{llcq} implies \eqref{ink11}.

In view of \eqref{neea}, for every $t\in[t_0,T_0]$ and $n\in\N_0$, we have
\begin{equation*}
\dist\big((s_n(t),x_n(t),u_n(t)),\E U\big)\leq\frac{1}{n}.
\end{equation*}
Passing to the limit as $n\to\infty$, we obtain $\dist\big((t,x(t),\tilde{u}(t)),\E U\big)=0$ for every $t\in[t_0,T_0]$. Therefore $\tilde{u}(t)\geq U(t,x(t))$ for all $t\in[t_0,T_0]$. Define $u(\tau):=\tilde{u}(t_0)-\int_{t_0}^\tau\eta(t)\,dt$. Then $u(\cdot)$ is an absolutely continuous function on $[t_0,T_0]$. Moreover, by \eqref{nc11},  we have $u(t)\geq \tilde{u}(t)$ for all $t\in[t_0,T_0]$. Hence it follows that $u(t)\geq \tilde{u}(t)\geq U(t,x(t))$ for all $t\in[t_0,T_0]$. Furthermore, by \eqref{ink11}, we get $(\dot{x}(t),\dot{u}(t))\in Q_{\bullet}(t,x(t),\tilde{u}(t))$ for a.e. $t\in[t_0,T_0]$. Since $(t,x(t))\in[0,T]\times\B_R$ for all $t\in[t_0,T_0]$, we get
$(\dot{x}(t),\dot{u}(t))\in Q_{\bullet}(t,x(t),\tilde{u}(t))=Q(t,x(t))$
for a.e. $t\in[t_0,T_0]$.
\end{proof}

\begin{Rem}
In the above proof we cannot replace condition  (L4) by (L5), because (L4) implies boundedness 
of sequence  $\{v_j\}_{j\in J},$ but (L5) does not. Condition (L5) plays an important role in the next proof.
\end{Rem}

\begin{proof}[Proof Theorem \ref{Twi5}]
Fix  $(t_0,x_0)\in \D U\cap [0,T)\times\R^{\scriptscriptstyle N}$.  We consider the family $\mathcal{F}$ of all pairs $[\,y(\cdot),[t_0,\tau)\,]$ such that an absolutely continuous function $y(\cdot)=(x,u)(\cdot)$ defined on the nondegenerate interval $[t_0,\tau)\subset[t_0,T]$ is a solution to
\begin{equation}\label{Twow-pc1}
\left\{\begin{array}{ll}
(\dot{x},\dot{u})(t)\in Q(t,x(t))\;\;\tn{for a.e.}\;t\in[t_0,\tau),\\[1mm]
(x,u)(t)\in\E U(t,\cdot)\;\,\tn{for all}\;t\in[t_0,\tau),\\[1mm]
(x,u)(t_0)=(x_0,U(t_0,x_0)).
\end{array}\right.
\end{equation}
In view of Theorem \ref{Twi5-l} the family $\mathcal{F}$ is nonempty. In the family $\mathcal{F}$ we define a partial order relation $[\,y(\cdot),[t_0,\tau)\,]\preccurlyeq[\,\bar{y}(\cdot),[t_0,\bar{\tau})\,]$ if and only if  $[t_0,\tau)\subset[t_0,\bar{\tau})$ and  $y(t)=\bar{y}(t)$ for all $t\in[t_0,\tau)$. We observe that in the family $\mathcal{F}$ every chain $\left\{\,[\,y_{\lambda}(\cdot),[t_0,\tau_{\lambda})\,]\,\right\}_{\lambda\in\Lambda}$ has an upper bound $[\,\bar{y}(\cdot),[t_0,\bar{\tau})\,]$, where $\bar{y}(t):=y_\lambda(t)$ for all $t\in[t_0,\tau_\lambda)$, $\lambda\in\Lambda$ and $[t_0,\bar{\tau}):=\bigcup_{\lambda\in\Lambda}[t_0,\tau_{\lambda})$. By the Kuratowski-Zorn Lemma in  $\mathcal{F}$ the exists a maximal element  $[\,y_{\bullet}(\cdot),[t_0,\tau_{\bullet})\,]$. 

We show that $\lim_{t\to\tau-}y_{\bullet}(t)$ exists and $y_{\bullet}(\cdot)$ is absolutely continuous on $[t_0,\tau_{\bullet}]$, where $y_{\bullet}(\tau_{\bullet}):=\lim_{t\to\tau_{\bullet}^-}y_{\bullet}(t)$, and $y_{\bullet}(\tau_{\bullet})\in\E U(\tau_{\bullet},\cdot)$. 
Indeed, let $y_{\bullet}(\cdot)=(x_{\bullet}(\cdot),u_{\bullet}(\cdot))$ and $t_n=\tau_{\bullet}-{\scriptstyle\frac{1}{2n}}(\tau_{\bullet}-t_0)$, $n\in\N$. Since $y_{\bullet}(\cdot)$ is a solution of \eqref{Twow-pc1}, we get $\dot{x}_{\bullet}(t)\in\D L(t,x_{\bullet}(t),\cdot)$ for a.e. $t\in[t_0,\tau_{\bullet})$. Hence, using (L5), we have $|\dot{x}_{\bullet}(t)|\leq c(t)(1+|x_{\bullet}(t)|)$ for a.e. $t\in[t_0,\tau_{\bullet})$. Since $x_{\bullet}(\cdot)$ is absolutely continuous on $[t_0,\tau_{\bullet})$, we have
\begin{equation*}
|x_{\bullet}(t)|\leq |x_0|+\int_{t_0}^{t_n}\!\!c(s)\,d\!s+\int_{t_0}^{t_n}\!\!c(s)\,|x_{\bullet}(s)|\,d\!s\;\;\; \tn{for all}\;\;\;t\in[t_0,t_n],\; n\in\N.
\end{equation*} 
By Gronwall’s Lemma
\begin{equation*}
|x_{\bullet}(t)|\leq\left(|x_0|+\int_{0}^T\!\!c(s)\,d\!s\right)\,\exp\left(\int_{0}^T\!\!c(s)\,d\!s\right)=:R\;\;\; \tn{for all}\;\;\;t\in[t_0,\tau_{\bullet}).
\end{equation*} 
The latter inequality, together with (L4), implies that $|\dot{x}_{\bullet}(t)|\leq C_R$ for a.e. $t\in[t_0,\tau_{\bullet})$ and some constant  $C_R$. Hence it follows that $\dot{x}_{\bullet}(\cdot)$ is integrable on $[t_0,\tau_{\bullet}]$. Thus the limit\linebreak $\lim_{t\to\tau-}x_{\bullet}(t)$ exists and the function $x_{\bullet}(\cdot)$ is absolutely continuous on  $[t_0,\tau_{\bullet}]$, where $x_{\bullet}(\tau_{\bullet}):=\lim_{t\to\tau_{\bullet}^-}x_{\bullet}(t)$. Since functions $L$ and $U$ are proper and lower semicontinuous, there exists a constant   $D>0$ such that  $-D\leq L(t,x,v)$ and $-D\leq U(t,x)$ for all $t\in[0,T]$, $x\in\B_R$, $v\in\B_{C_R}$. Since $y_{\bullet}(\cdot)$ is a solution of \eqref{Twow-pc1}, we get $-D\leq L(t,x_{\bullet}(t),\dot{x}_{\bullet}(t))\leq -\dot{u}_{\bullet}(t)$, so $\dot{u}_{\bullet}(t)\leq D$ for a.e. $t\in[t_0,\tau_{\bullet})$. Moreover,  $-D\leq U(t_n,x_{\bullet}(t_n))\leq u_{\bullet}(t_n)$ for every $n\in\N$. Since $u_{\bullet}(\cdot)$ is absolutely continuous on $[t_0,\tau_{\bullet})$,  we obtain, for all $n\in\N$,
\begin{eqnarray*}
\int_{t_0}^{t_\bullet}\chi_{[t_0,t_n]}(t)\,|\dot{u}_{\bullet}(t)|\,dt &=& \int_{t_0}^{t_n}2\max\{\dot{u}_{\bullet}(t),0\}\,dt-\int_{t_0}^{t_n}\dot{u}_{\bullet}(t)\,dt\\
&\leq & 2TD+u_{\bullet}(t_0)-u_{\bullet}(t_n)\;\;\leq\;\; 2TD+|U(t_0,x_0)|+D.
\end{eqnarray*}
Due to the above inequality and Lebesgue’s Monotone Convergence Theorem we obtain
that $\dot{u}_{\bullet}(\cdot)$ is integrable on $[t_0,\tau_{\bullet}]$. Hence it follows that the limit  $\lim_{t\to\tau-}u_{\bullet}(t)$ exists and $u_{\bullet}(\cdot)$ is absolutely continuous on $[t_0,\tau_\bullet]$, where $u_\bullet(\tau_\bullet):=\lim_{t\to\tau_\bullet^-}u_\bullet(t)$. In view of \eqref{Twow-pc1} we have $u_{\bullet}(t_n)\geq U(t_n,x_{\bullet}(t_n))$ for all $n\in\N$. The latter inequality, together with lower semicontinuity of $U$, implies $u_{\bullet}(\tau_{\bullet})\geq U(\tau_{\bullet},x_{\bullet}(\tau_{\bullet}))$. So $y_{\bullet}(\tau_{\bullet})\in\E U(\tau_{\bullet},\cdot)$.

The proof is completed by showing that $\tau_{\bullet}=T$. Suppose, contrary to our claim, that $\tau_{\bullet}<T$. Then, in view of Theorem \ref{Twi5-l}, there exist  $\bar{\tau}_{\bullet}\in(\tau_{\bullet},T)$ and  $\bar{y}(\cdot)=(\bar{x},\bar{u})(\cdot)$ such that
\begin{equation*}
\left\{\begin{array}{ll}
\dot{\bar{y}}(t)\in Q(t,\bar{x}(t))\;\;\tn{p.w.}\;t\in[\tau_{\bullet},\bar{\tau}_{\bullet}),\\[1mm]
\bar{y}(t)\in \E U(t,\cdot),\;\;\forall\;t\in[\tau_{\bullet},\bar{\tau}_{\bullet}),\\[1mm]
\bar{y}(\tau_{\bullet})=y_{\bullet}(\tau_{\bullet}).
\end{array}\right.
\end{equation*}
We define the function $\bar{y}_{\bullet}(\cdot)$ by the formula $\bar{y}_{\bullet}(t):=y_{\bullet}(t)$ for all  $t\in[t_0,\tau_{\bullet})$ and $\bar{y}_{\bullet}(t):=\bar{y}(t)$ for all $t\in[\tau_{\bullet},\bar{\tau}_{\bullet})$. We observe that $[\,\bar{y}_{\bullet}(\cdot),[t_0,\bar{\tau}_{\bullet})\,]\in \mathcal{F}$  and   $[\,y_{\bullet}(\cdot),[t_0,\tau_{\bullet})\,]\prec[\,\bar{y}_{\bullet}(\cdot),[t_0,\bar{\tau}_{\bullet})\,]$, which contradicts the definition of the maximal element $[\,y_{\bullet}(\cdot),[t_0,\tau_{\bullet})\,]$.
\end{proof}

\section{Hamilton-Jacobi-Bellman Theory}\label{section-5}

\noindent We define the functional $\Gamma[\,\cdot\,]$ for any $x(\cdot)\in\mathcal{A}\left([t_0,T],\R^{\scriptscriptstyle N}\right)$  by the formula
\begin{equation*}
\Gamma[x(\cdot)]=g(x(T))+\int_{t_0}^TL(t,x(t),\dot{x}(t))\,dt.
\end{equation*}
Then the value function for any $(t_0,x_0)\in[0,T]\times\R^{\scriptscriptstyle N}$ is given by
\begin{equation*}
V(t_0,x_0)= \left\{
\begin{array}{lcl}
\inf\left\{\,\Gamma[x(\cdot)]\,\mid\, x(\cdot)\in\mathcal{A}\left([t_0,T],\R^{\scriptscriptstyle N}\right)\!,\, x(t_0)=x_0\right\} & \tn{if} & t_0\in[0,T), \\[1mm]
g(x_0) & \tn{if} & t_0=T.
\end{array}
\right.
\end{equation*}

\begin{Th}\label{vfgr}
Assume that $L$ satisfies \tn{(L1)-(L5)} and $g$ is a proper, lower semicontinuous function. Let $V$ be the value function associated with $L$ and $g$. Then we have the following.
\begin{enumerate}
\item[\tn{\bf{(a)}}] $V$ is a proper, lower semicontinuous function.\vspace{1mm}
\item[\tn{\bf{(b)}}] For every $(t_0,x_0)\in\D V\cap [0,T)\times\R^{\scriptscriptstyle N}$ there exists $\bar{x}(\cdot)\in\mathcal{A}([t_0,T],\R^{\scriptscriptstyle N})$ such that $V(t_0,x_0)=\Gamma[\bar{x}(\cdot)]$ and $\bar{x}(t_0)=x_0$. 
\end{enumerate}
\end{Th}

The theorem is a consequence of the results from \cite[Thm. 7.6]{AM} and \cite[Thm. 6.3]{AM1}. It is easy to note that condition (HLC) in \cite[Thm. 7.6]{AM} and \cite[Thm. 6.3]{AM1} can be weakened to the condition: for every $R>0$ there exists an integrable function $k_R:[0,T]\to[0,\infty)$ such that $L(t,x,v)\geq -k_R(t)$ for all $x\in\B_R$, $v\in\R^{\scriptscriptstyle N}$ and a.e. $t\in[0,T]$. By (L4) there exists a constant $C_R\geq 0$ such that  $\|\D L(t,x,\cdot)\|\leq C_R$ for all $(t,x)\in[0,T]\times\B_R$. Since $L$ is proper and lower semicontinuous (due to (L1)-(L2)), there exists a constant $k_R>0$ such that $L(t,x,v)\geq -k_R$ for all  $t\in[0,T]$, $x\in\B_R$, $v\in\B_{C_R}$. Thus $L(t,x,v)\geq -k_R$ for all  $t\in[0,T]$, $x\in\B_R$, $v\in\R^{\scriptscriptstyle N}$. It means that the condition (L6) is not necessary in the above theorem.

\vspace{2mm}
In the Subsection \ref{els} we show that the value function $V$, associated with $H^{\ast}$ and $g$, is  a lower semicontinuous solution of \eqref{rowhj}, provided that $H$ satisfies (H1)-(H4) and $g$ is proper and lower semicontinuous. We will see that the condition (H5) is not required to prove the above fact. In the Subsection \ref{uls}, assuming additionally that $H$ satisfies (H5), we show that the value function $V$ is a unique lower semicontinuous solution of \eqref{rowhj}. It turns out that the condition (H5) is  necessary in some sense to prove the uniqueness  result. More precisely, one can propose two different lower semicontinuous solutions of \eqref{rowhj} with the Hamiltonian $H$ satisfying (H1)-(H4); see \cite{AM0}.

\subsection{Existence of Lower Semicontinuous Solutions}\label{els}
\begin{Th}\label{lfm-Prop}
Assume that $L$ satisfies \tn{(L1)-(L5)} and $g$ is proper, lower semicontinuous.\linebreak If $V$ is the value function associated with $L$ and $g$, then for all $(t_0,x_0)\in\D V\cap [0,T)\times\R^{\scriptscriptstyle N}$ there exists $v_0\in\R^{\scriptscriptstyle N}$ such that 
\begin{equation}\label{FW-IS111}
d V(t_0,x_0)(1,v_0)\leq -L(t_0,x_0,v_0).
\end{equation}
\end{Th}

\begin{proof}
Fix $(t_0,x_0)\in\D V\cap [0,T)\times\R^{\scriptscriptstyle N}$. By Theorem \ref{vfgr} there exists $\bar{x}(\cdot)\in\mathcal{A}([t_0,T],\R^{\scriptscriptstyle N})$ such that $V(t_0,x_0)=\Gamma[\bar{x}(\cdot)]$ and $\bar{x}(t_0)=x_0$. We define the absolutely continuous function $\bar{u}:[t_0,T]\to\R$ by the formula
\begin{equation*}
\bar{u}(t):=g(\bar{x}(T))+\int_t^TL(s,\bar{x}(s),\dot{\bar{x}}(s))\;d\!s.
\end{equation*}
Observe that $\bar{u}(t_0)=V(t_0,x_0)$, $\bar{u}(t)\geq V(t,\bar{x}(t))$ for all $t\in[t_0,T]$ and $\dot{\bar{u}}(t)=-L(t,\bar{x}(t),\dot{\bar{x}}(t))$ for a.e. $t\in[t_0,T]$. Since $V(t_0,x_0)<+\infty$, we have $\dot{\bar{x}}(t)\in\D L(t,\bar{x}(t),\cdot)$ for a.e. $t\in[t_0,T]$. Therefore, by (L4),  we obtain $|\dot{\bar{x}}(t)|\leq C_R$ for a.e. $t\in[t_{0},T]$, where $R:=\max_{t\in[t_0,T]}|\bar{x}(t)|$.  The latter inequality, together with absolute continuity of $\bar{x}(\cdot)$, implies that $\bar{x}(\cdot)$ is Lipschitz continuous. Hence there exist $h_n\rightarrow 0+$ and $v_0\in\R^{\scriptscriptstyle N}$ such that
\begin{equation}\label{lfm-0}
\frac{1}{h_n}\int_{t_0}^{t_0+h_n}\dot{\bar{x}}(t)\;dt=\frac{\bar{x}(t_0+h_n)-\bar{x}(t_0)}{h_n}\rightarrow v_0.
\end{equation}
We choose $v_{n}$ such that  $\bar{x}(t_0+h_n)=x_0+v_nh_n$. Then we have $(1,v_n)\rightarrow(1,v_0)$. Moreover
\begin{eqnarray}\label{lfm-1}
\nonumber d V(t_0,x_0)(1,v_0) &\leq & \liminf_{n\to\infty}\frac{V((t_0,x_0)+h_n(1,v_n))-V(t_0,x_0)}{h_n}\\
\nonumber &= & \liminf_{n\to\infty}\frac{V(t_0+h_n,\bar{x}(t_0+h_n))-V(t_0,x_0)}{h_n}\\
\nonumber &\leq & \liminf_{n\to\infty}\frac{\bar{u}(t_0+h_n)-\bar{u}(t_0)}{h_n}\\
\nonumber &= & \liminf_{n\to\infty}\frac{1}{h_n}\int_{t_0}^{t_0+h_n}\dot{\bar{u}}(t)\;dt\\
 &\leq & \liminf_{n\to\infty}\frac{1}{h_n}\int_{t_0}^{t_0+h_n}-L(t,\bar{x}(t),\dot{\bar{x}}(t)))\;dt.
\end{eqnarray}
We consider two cases:

\bf{Case 1.} If $d V(t_0,x_0)(1,v_0)=-\infty$, then the inequality \eqref{FW-IS111} holds obviously.  

\bf{Case 2.} If $d V(t_0,x_0)(1,v_0)>-\infty$, then denoting by $\eta_0$  the right hand side of  \eqref{lfm-1} we obtain $\eta_0>-\infty$. Since $L$ is proper and lower semicontinuous, there exists a constant $D>0$ such that $L(t,x,v)\geq -D$ for all $t\in[0,T]$, $x\in\B_R$, $v\in\B_{C_R}$. Hence, $-L(t,\bar{x}(t),\dot{\bar{x}}(t))\leq D$ for a.e. $t\in[t_{0},T]$. The latter inequality yields  $\eta_0<+\infty$. Therefore $\eta_0$ is a real number. Fix $\varepsilon>0$ and choose $n_0\in\N$ such that for every $n>n_0$  the following property holds
\begin{equation}\label{lfm-2}
\forall\:t\in[t_0,t_0+h_n],\;\;\;\; |t_0-t|<\varepsilon,
\;\;\;|x_0-\bar{x}(t)|<\varepsilon.
\end{equation}
We observe that for a.e. $t\in[t_0,T]$ we have
$$(\dot{\bar{x}}(t),-L(t,\bar{x}(t),\dot{\bar{x}}(t)))\in Q(t,\bar{x}(t)),$$
so using \eqref{lfm-2} for a.e. $t\in[t_0,t_0+h_n]$  we have
$$(\dot{\bar{x}}(t),-L(t,\bar{x}(t),\dot{\bar{x}}(t)))\in Q(t_0,x_0;\varepsilon).$$
Therefore for a.e. $t\in[t_0,t_0+h_n]$ we have
$$(\dot{\bar{x}}(t),-L(t,\bar{x}(t),\dot{\bar{x}}(t)))\in \mathrm{cl}\,\mathrm{conv}\,Q(t_0,x_0;\varepsilon),$$
hence for every $n>n_0$ we have
$$\frac{1}{h_n}\int_{t_0}^{t_0+h_n}\big(\dot{\bar{x}}(t),-L(t,\bar{x}(t),\dot{\bar{x}}(t))\big)\,dt\in
\mathrm{cl}\,\mathrm{conv}\,Q(t_0,x_0;\varepsilon).$$
Passing to a subsequence as $n\rightarrow\infty$ we have
$$(v_0,\eta_0)\in\mathrm{cl}\,\mathrm{conv}\,Q(t_0,x_0;\varepsilon).$$
Since $\varepsilon$ is arbitrary, we obtain 
$$(v_0,\eta_0)\in \bigcap_{\varepsilon>0}\mathrm{cl}\,\mathrm{conv}\,Q(t_0,x_0;\varepsilon).$$
By Lemma \ref{llcq} we get $(v_0,\eta_0)\in Q(t_0,x_0)$. The latter, together with \eqref{lfm-1}, implies  
$$d V(t_0,x_0)(1,v_0)\leq\eta_0\leq-L(t_0,x_0,v_0).$$
Thus the inequality \eqref{FW-IS111} also holds in this case.
\end{proof}

\begin{Rem}
We cannot strengthen the conclusion of Theorem \ref{lfm-Prop} by requiring that \linebreak $v_0\in\D L(t_0,x_0,\cdot)$. Indeed, let $L$ and $V$ be given by \eqref{expnun1} and \eqref{expnun2} respectively. Suppose that \eqref{FW-IS111} holds for some $v_0\in\D L(t_0,x_0,\cdot)$, then $+\infty=d V(\xi,\xi)(1,0)\leq -L(\xi,\xi,0)=0$ for all $\xi\in(0,T)$, which is impossible.
\end{Rem}

\begin{Lem}\label{cdwt-01}
Assume that $L$ satisfies \tn{(L1)-(L3)}. Let $(t_0,x_0,u_0)\in(0,T]\times\R^{\scriptscriptstyle N}\times\R$ and\linebreak $v_0\in\D L(t_0,x_0,\cdot)$. Then there exist $\tau>0$ and a function $(x,u):[t_0-\tau,t_0]\to\R^{\scriptscriptstyle N}\times\R$\linebreak of class $\mathcal{C}^1$ with $(x,u)(t_0)=(x_0,u_0)$ which satisfies  $(\dot{x}(t),\dot{u}(t))\in Q(t,x(t))$ for all $t\in[t_0-\tau,t_0]$  and  $(\dot{x},\dot{u})(t_0^-)=(v_0,-L(t_0,x_0,v_0))$.
\end{Lem}

\begin{proof}
Fix $(t_0,x_0,u_0)\in(0,T]\times\R^{\scriptscriptstyle N}\times\R$ and $v_0\in\D L(t_0,x_0,\cdot)$. From Corollaty \ref{wrow-wm} it follows that $Q$ satisfies $(\mathcal{Q}1)$-$(\mathcal{Q}3)$. In view of Theorem \ref{th-oparam} there exists a continuous function $\e(\cdot,\cdot,\cdot)$ satisfying (P1) and (P2). Let us fix $a_0:=(v_0,-L(t_0,x_0,v_0))\in Q(t_0,x_0)$. We define a continuous function $\bar{\e}(\cdot,\cdot,\cdot)$ for every $(t,x,u)\in[0,T]\times\R^{\scriptscriptstyle N}\times\R$ by 
$$\bar{\e}(t,x,u):=\,\e(t,x,a_0).$$
In view of Peano Theorem there exist $\tau>0$ and the function $(x,u):[t_0-\tau,t_0]\to\R^{\scriptscriptstyle N}\times\R$ of class $\mathcal{C}^1$ which is a solution to
\begin{displaymath}
\left\{
\begin{array}{l}
(\dot{x},\dot{u})(t)= \bar{\e}(t,x(t),u(t))\;\;\tn{for all}\;\;t\in[t_0-\tau,t_0],\\
(x,u)(t_0)=(x_0,u_0).
\end{array}
\right.
\end{displaymath}
By (P1) for all $t\in[t_0-\tau,t_0]$ we have
\begin{eqnarray*}
(\dot{x},\dot{u})(t)&=&\bar{\e}(t,x(t),u(t))\\
&=&\e(t,x(t),a_0)\;\in\; Q(t,x(t)).
\end{eqnarray*}
Since $a_0\in Q(t_0,x_0)$, from (P2) we deduce
\begin{eqnarray*}
(\dot{x},\dot{u})(t_0^-) &=&  \bar{\e}(t_0,x(t_0),u(t_0))\\
&=&\bar{\e}(t_0,x_0,u_0)\;=\;\e(t_0,x_0,a_0)\\
&=& a_0\;=\;(v_0,-L(t_0,x_0,v_0)).
\end{eqnarray*}
This finishes the proof.
\end{proof}

\begin{Th}\label{olr-P1}
Assume that $L$ satisfies \tn{(L1)-(L5)} and $g$ is proper, lower semicontinuous.\linebreak If $V$ is the value function associated with $L$ and $g$, then for all $(t_0,x_0)\in\D V\cap (0,T]\times\R^{\scriptscriptstyle N}$  and all $v_0\in\D L(t_0,x_0,\cdot)$ we have
\begin{equation}\label{olr-n1}
d V(t_0,x_0)(-1,-v_0)\leq L(t_0,x_0,v_0).
\end{equation}
\end{Th}

\begin{proof}
Fix $(t_0,x_0)\in\D V\cap (0,T]\times\R^{\scriptscriptstyle N}$ and  $v_0\in\D L(t_0,x_0,\cdot)$.
In view of Lemma \ref{cdwt-01} there exist $\tau>0$ and a function $(x_1,u_1):[t_0-\tau,t_0]\to\R^{\scriptscriptstyle N}\times\R$ of class $C^1$ with $(x_1,u_1)(t_0)=(x_0,V(t_0,x_0))$ which satisfies  $(\dot{x}_1(t),\dot{u}_1(t))\in Q(t,x_1(t))$ for all $t\in[t_0-\tau,t_0]$ and $(\dot{x}_1,\dot{u}_1)(t_0^-)=(v_0,-L(t_0,x_0,v_0))$.
Let $0<h_n<\tau$ and $h_n\rightarrow 0$. Then, we choose a sequence $v_n$ such that $x_1(t_0-h_n)=x_0-h_nv_n$ and $v_n\rightarrow v_0$. By Theorem \ref{vfgr} (b) there exists  $x_2(\cdot)\in\mathcal{A}([t_0,T],\R^{\scriptscriptstyle N})$ such that $V(t_0,x_0)=\Gamma[x_2(\cdot)]$ and $x_2(t_0)=x_0$.  We define an absolutely continuous function $u_2:[t_0,T]\to\R$ by the formula
\begin{equation*}
u_2(t):=g(x_2(T))+\int_t^TL(s,x_2(s),\dot{x}_2(s))\;d\!s.
\end{equation*}
Then, we can define an absolutely continuous function  $(\bar{x},\bar{u})(\cdot)$ on $[t_0-\tau,T]$ by 
\begin{equation*}
(\bar{x},\bar{u})(t):=\left\{
\begin{array}{ccl}
(x_1,u_1)(t) & \tn{if} & t\in[t_0-\tau,t_0], \\[1mm]
(x_2,u_2)(t) & \tn{if} & t\in[t_0,T].
\end{array}
\right.
\end{equation*}
Observe that $\bar{u}(t_0)=V(t_0,x_0)$, $\bar{u}(T)=g(\bar{x}(T))$ and $(\dot{\bar{x}},\dot{\bar{u}})(t)\in Q(t,\bar{x}(t))$ for a.e. $t\in[t_0-\tau,T]$. Hence, for all $t\in[t_0-\tau,T]$, we have
$$\bar{u}(t)\,=\,\bar{u}(T)+\int_t^T\!\!-\dot{\bar{u}}(s)\;d\!s\,\geq\, g(\bar{x}(T))+\int_t^TL(s,\bar{x}(s),\dot{\bar{x}}(s))\;d\!s\,\geq\, V(t,\bar{x}(t)).$$
Since $V(t,\bar{x}(t))\leq\bar{u}(t)$ for all $t\in[t_0-\tau,T]$ and $V(t_0,x_0)=\bar{u}(t_0)$, we obtain
\begin{eqnarray*}
d V(t_0,x_0)(-1,-v_0) &\leq & \liminf_{n\to\infty}\frac{V((t_0,x_0)+h_n(-1,-v_n))-V(t_0,x_0)}{h_n}\\
 &= & \liminf_{n\to\infty}\frac{V(t_0-h_n,x_1(t_0-h_n))-V(t_0,x_0)}{h_n}\\
 &= & \liminf_{n\to\infty}\frac{V(t_0-h_n,\bar{x}(t_0-h_n))-V(t_0,x_0)}{h_n}\\
 &\leq & \lim_{n\to\infty}\frac{\bar{u}(t_0-h_n)-\bar{u}(t_0)}{h_n}\\
 &=& \lim_{n\to\infty}\frac{u_1(t_0-h_n)-u_1(t_0)}{h_n}\\ 
&= & -\dot{u}_1(t_0^-)\;\;=\;\; L(t_0,x_0,v_0),
\end{eqnarray*}
which completes the proof.
\end{proof}

\begin{Th}[Existence]\label{tw-ist-hjb}
Assume that $H$ satisfies \tn{(H1)-(H4)} and $g$ is proper and lower semicontinuous.
If $V$ is the value function associated with $H^{\ast}$ and $g$, then $V$ is  a lower semicontinuous solution of \eqref{rowhj}. 
\end{Th}

\begin{proof}
Let $L(t,x,\cdot\,)=H^{\ast}(t,x,\cdot\,)$. In view of Proposition \ref{prop2-fmw} $L$ satisfies (L1)-(L5).

Fix $(t,x)\in\D V\cap [0,T)\times\R^{\scriptscriptstyle N}$. Then by Theorem \ref{lfm-Prop} there exists  $\bar{v}\in\R^{\scriptscriptstyle N}$ such that
\begin{equation}\label{tw-ist-hjb-1}
d V(t,x)(1,\bar{v})\,\leq\, -L(t,x,\bar{v}).
\end{equation}
Let  $(p_t,p_x)\in\partial V(t,x)$. Then, by the definition of the subdifferential,  we deduce that
\begin{equation}\label{tw-ist-hjb-2}
\langle\,(p_t,p_x),\,(1,\bar{v})\,\rangle\,\leq\, d V(t,x)(1,\bar{v}).
\end{equation}
Combining inequalities  \eqref{tw-ist-hjb-1} and \eqref{tw-ist-hjb-2} we obtain that
$$p_t+\langle p_x,\bar{v}\rangle\,\leq\, -L(t,x,\bar{v}).$$
The latter inequality, together with $H(t,x,\cdot\,)=L^{\ast}(t,x,\cdot\,)$, implies that
$$-p_t+H(t,x,-p_x)\,\geq\,-p_t+\langle -p_x,\bar{v}\rangle-L(t,x,\bar{v})\,\geq\,0.$$ 
Therefore the inequality \eqref{lsc-solutions-1} holds.

Fix $(t,x)\in\D V\cap (0,T]\times\R^{\scriptscriptstyle N}$. Then, by Theorem \ref{olr-P1}, we have
\begin{equation}\label{tw-ist-hjb-3}
\forall\,v\in\D L(t,x,\cdot),\quad d V(t,x)(-1,-v)\,\leq\, L(t,x,v).
\end{equation}
Let  $(p_t,p_x)\in\partial V(t,x)$. Then, by the definition of the subdifferential, we deduce that
\begin{equation}\label{tw-ist-hjb-4}
\forall\,v\in\D L(t,x,\cdot),\quad\langle(p_t,p_x),(-1,-v)\rangle\leq d V(t,x)(-1,-v).
\end{equation}
Combining inequalities \eqref{tw-ist-hjb-3} and \eqref{tw-ist-hjb-4}  we obtain
$$\forall\,v\in\D L(t,x,\cdot),\quad -p_t-\langle p_x,v\rangle\leq L(t,x,v).$$
The latter inequality, together with $H(t,x,\cdot\,)=L^{\ast}(t,x,\cdot\,)$, implies that
$$-p_t+H(t,x,-p_x)= -p_t+\sup\nolimits_{\;v\in\R^{\scriptscriptstyle N}}\{\langle -p_x,v\rangle-L(t,x,v)\}\leq 0.$$ 
Therefore the inequality \eqref{lsc-solutions-2} holds.
\end{proof}

\subsection{Uniqueness of Lower Semicontinuous Solutions}\label{uls}

\begin{Prop}\label{ws-11}
Let $U$ be a proper and lower semicontinuous function. Assume that $H$ satisfies \tn{(H1)-(H2)}. If for every $(t,x)\in \D U\cap (0,T]\times\R^{\scriptscriptstyle N}$ and every $(p_{t},p_{x})\in\partial U(t,x)$ the inequality \eqref{lsc-solutions-2} holds, then the condition \eqref{mwn} also holds. 
\end{Prop}

\begin{proof}
Let $L(t,x,\cdot\,)=H^{\ast}(t,x,\cdot\,)$. In view of Proposition \ref{prop2-fmw} the Lagrangian $L$ satisfies (L1)-(L3). Let us fix $(t,x)\in\D U\cap (0,T]\times\R^{\scriptscriptstyle N}$.

Let  $(n^t\!,n^x\!\!,n^u)\!\in\! N_{\E U}(t,x,U(t,x))$ and $n^u\!<\!0$. Then, by the definition of the normal cone, 
$$\left(n^t/|n^u|,n^x/|n^u|,-1\right)\in N_{\E U}(t,x,U(t,x)).$$
Due to relation between normal cones and subdifferentials, from Subsection \ref{nap}, we get
$$(n^t/|n^u|,n^x/|n^u|)\in\partial U(t,x).$$ 
The latter, together with \eqref{lsc-solutions-2}, implies that
$$-n^t/|n^u|+H(t,x,-n^x/|n^u|)\,\leq\, 0.$$
Since $H(t,x,\cdot\,)=L^{\ast}(t,x,\cdot\,)$, for all   $v\in \D L(t,x,\cdot)$ we have 
$$-n^t/|n^u|+\langle\, -n^x/|n^u|,v\, \rangle -L(t,x,,v)\,\leq\, 0. $$
By multiplying both sides of the above inequality by $-|n^u|$, we get
$$n^t+\langle n^x,v \rangle -n^uL(t,x,v)\,\geq\, 0, $$
which completes the proof of \eqref{mwn} in this case.

Let $(n^t,n^x,0)\in N_{\E U}(t,x,U(t,x))$. In view of Lemma \ref{lemrlsc} there exist $(t_k,x_k)\rightarrow (t,x)$\linebreak  and $(n_k^t,n_k^x,n_k^u)\rightarrow(n^t,n^x,0)$ satisfying $n_k^u<0$ and
$(n_k^t,n_k^x,n_k^u)\in N_{\E U}(t_k,x_k,U(t_k,x_k))$\linebreak for all $k\in\N$. Then, by the definition of a normal cone, for all $k\in\N$ we have
$$\left(n_k^t/|n_k^u|,\,n_k^x/|n_k^u|,\,-1\right)\in N_{\E U}(t_k,x_k,U(t_k,x_k)).$$ 
Due to relation between normal cones and subdifferentials, from Subsection \ref{nap}, we get 
$$(n_k^t/|n_k^u|,n_k^x/|n_k^u|)\in \partial U(t_k,x_k),$$ 
for all $k\in\N$. The latter, together with \eqref{lsc-solutions-2}, implies that, for all large $k\in\N$,
$$-n_k^t/|n_k^u|+H(t_k,x_k,-n_k^x/|n_k^u|)\,\leq\, 0.$$
Since $H(t,x,\cdot\,)=L^{\ast}(t,x,\cdot\,)$, we have, for all   $v\in \D L(t_k,x_k,\cdot)$ and all large $k\in\N$,
$$-n_k^t/|n_k^u|+\langle -n_k^x/|n_k^u|,v\, \rangle -L(t_k,x_k,v)\,\leq\, 0.$$
By multiplying both sides of the above inequality by $-|n_k^u|$, we get
\begin{equation}\label{r4n0}
n_k^t+\langle\, n_k^x,v\, \rangle -n_k^uL(t_k,x_k,v)\,\geq\, 0, 
\end{equation}
for all   $v\in \D L(t_k,x_k,\cdot)$ and all large $k\in\N$. Let us fix $\bar{v}\in\D L(t,x,\cdot)$. By (L3)  there exists $v_k\to\bar{v}$ such that $L(t_k,x_k,v_k)\to L(t,x,\bar{v})$. Hence $v_k\in\D L(t_k,x_k,\cdot)$  for all large $k\in\N$. If we set $v:=v_k$ in the inequality \eqref{r4n0} and pass to the limit as $k\to\infty$, then,
$$n^t+\langle\, n^x,\bar{v}\, \rangle\,\geq \,0,$$
which completes the proof of \eqref{mwn}. 
\end{proof}

\begin{Prop}\label{flw1}
Let $U$ be a proper and lower semicontinuous function. Assume that $H$ satisfies \tn{(H1)-(H3)}. If for every $(t,x)\in \D U\cap [0,T)\times\R^{\scriptscriptstyle N}$ and every $(p_{t},p_{x})\in\partial U(t,x)$ the inequality \eqref{lsc-solutions-1} holds, then the condition \eqref{swn} also holds.
\end{Prop}

\begin{proof} Let $L(t,x,\cdot\,)=H^{\ast}(t,x,\cdot\,)$. In view of Proposition \ref{prop2-fmw} the Lagrangian $L$ satisfies (L1)-(L4). Let us fix $(t,x)\in\D U\cap [0,T)\times\R^{\scriptscriptstyle N}$.

Let  $(n^t\!,n^x\!\!,n^u)\!\in\! N_{\E U}(t,x,U(t,x))$ and $n^u\!<\!0$.  Then, by the definition of the normal cone,
$$\left(n^t/|n^u|,\,n^x/|n^u|,\,-1\right)\in N_{\E U}(t,x,U(t,x)).$$
Due to relation between normal cones and subdifferentials, from Subsection \ref{nap}, we get
$$(n^t/|n^u|,n^x/|n^u|)\in\partial U(t,x).$$ 
The latter, together with \eqref{lsc-solutions-1}, implies that
$$-n^t/|n^u|+H(t,x,-n^x/|n^u|)\,\geq\, 0.$$
Since $H(t,x,\cdot\,)=L^{\ast}(t,x,\cdot\,)$ and (L4) holds, there exists $v\in \D L(t,x,\cdot)$ such that 
$$-n^t/|n^u|+\langle -n^x/|n^u|,v\, \rangle -L(t,x,v)\,\geq\, 0. $$
By multiplying both sides of the above inequality by $-|n^u|$, we have
$$n^t+\langle n^x,v \rangle -n^uL(t,x,U(t,x),v)\leq 0.$$
If we set $t_k:=t$, $x_k:=x$, $v_k:=v$, $\alpha_k:=0$, then we obtain \eqref{swn} in this case.

Let $(n^t,n^x,0)\in N_{\E U}(t,x,U(t,x))$. In view of Lemma \ref{lemrlsc} there exist $(t_k,x_k)\rightarrow (t,x)$\linebreak  and $(n_k^t,n_k^x,n_k^u)\rightarrow(n^t,n^x,0)$ satisfying $n_k^u<0$ and
$(n_k^t,n_k^x,n_k^u)\in N_{\E U}(t_k,x_k,U(t_k,x_k))$\linebreak for all $k\in\N$. Then, by the definition of a normal cone, for all $k\in\N$ we have
$$\left(n_k^t/|n_k^u|,\,n_k^x/|n_k^u|,\,-1\right)\in N_{\E U}(t_k,x_k,U(t_k,x_k)).$$ 
Due to relation between normal cones and subdifferentials, from Subsection \ref{nap}, we get 
$$(n_k^t/|n_k^u|,n_k^x/|n_k^u|)\in \partial U(t_k,x_k),$$ 
for all $k\in\N$. The latter, together with \eqref{lsc-solutions-1}, implies that, for all large $k\in\N$,
$$-n_k^t/|n_k^u|+H(t_k,x_k,-n_k^x/|n_k^u|)\,\geq\, 0.$$
Since $H(t,x,\cdot\,)=L^{\ast}(t,x,\cdot\,)$ and (L4) holds, there exists $v_k\in \D L(t_k,x_k,\cdot)$ such that
$$-n_k^t/|n_k^u|+\langle -n_k^x/|n_k^u|,v_k \rangle -L(t_k,x_k,v_k)\,\geq\, 0. $$
By multiplying both sides of the above inequality by $-|n_k^u|$, we get
\begin{equation}\label{snode}
n_k^t+\langle n_k^x,v_k \rangle +|n_k^u|\,L(t_k,x_k,v_k)\,\leq\, 0.
\end{equation}
Let $R:=\sup\{|x_k|\mid k\in\N\}$. Then, by (L4), we obtain $|v_k|\leq C_R$ for all large $k\in\N$.
Since $L$ is proper and lower semicontinuous (due to (L1)-(L2)), there exists a constant $D>0$ such that $L(t,x,v)\geq -D$ for all $t\in[0,T]$, $x\in\B_R$, $v\in\B_{C_R}$. Hence $L(t_k,x_k,v_k)\geq-D$ for all $k\in\N$. The latter, together with \eqref{snode}, implies that, for all large $k\in\N$,
$$n_k^t+\langle\, n_k^x,v_k\, \rangle \,\leq\, -|n_k^u|\,L(t_k,x_k,v_k)\,\leq\,|n_k^u|\,D.$$
Set $a_k:=n^t-n_k^t$, $b_k:=n^x-n_k^x$, $c_k:=|n_k^u|\,D$. We observe that sequences $\{a_k\}$, $\{b_k\}$ and $\{c_k\}$ converge to $0$ as $k\to\infty$. Moreover we derive that
\begin{eqnarray*}
n^t+\langle v_k, n^x \rangle & = & a_k+n_k^t+\langle v_k,b_k+n_k^x \rangle\,=\,a_k+\langle v_k,b_k \rangle +n_k^t +\langle v_k,n_k^x \rangle\\
& \leq &  a_k+|v_k||b_k|+c_k\;\;\leq\;\;a_k+C_R|b_k|+c_k\;\;=:\alpha_k. 
\end{eqnarray*}
Obviously $\alpha_k\rightarrow 0$ as $k\to\infty$, which completes the proof of \eqref{swn}. 
\end{proof}

\begin{Th}[Uniqueness]\label{tw-jed-hjb}
Assume that $H$ satisfies \tn{(H1)-(H5)} and $g$ is proper and lower semicontinuous.
Let $V$ be the value function associated with $H^{\ast}$ and $g$. If $U$ is a lower semicontinuous solution of \eqref{rowhj}, then $U=V$ on $[0,T]\times\R^{\scriptscriptstyle N}$.
\end{Th}

\begin{proof}
Let $L(t,x,\cdot)\!=\!H^{\ast}(t,x,\cdot)$. By Proposition \ref{prop2-fmw} and Theorem \ref{tw2_rlhmh}  $L$ satisfies (L1)-(L6). 

We first show that $U\leq V$ on $[0,T]\times\R^{\scriptscriptstyle N}$. Obviously $U(T,x_0)=g(x_0)=V(T,x_0)$ for all $x_0\in\R^{\scriptscriptstyle N}$. Let us fix $(t_0,x_0)\in\D V\cap [0,T)\times\R^{\scriptscriptstyle N}$. In view of Theorem \ref{vfgr} (b) there exists $\bar{x}(\cdot)\in\mathcal{A}([t_0,T],\R^{\scriptscriptstyle N})$ such that $V(t_0,x_0)=\Gamma[\bar{x}(\cdot)]$ and $\bar{x}(t_0)=x_0$. 
We define an absolutely continuous function $\bar{u}:[t_0,T]\to\R$ by the formula
\begin{equation*}
\bar{u}(t):=g(\bar{x}(T))+\int_t^TL(s,\bar{x}(s),\dot{\bar{x}}(s))\;d\!s.
\end{equation*}
Observe that $\bar{u}(t_0)=V(t_0,x_0)$, $\bar{u}(T)=g(\bar{x}(T))$ and $\dot{\bar{u}}(t)=-L(t,\bar{x}(t),\dot{\bar{x}}(t))$ for a.e. $t\in[t_0,T]$. Hence $(\dot{\bar{x}},\dot{\bar{u}})(t)\in Q(t,\bar{x}(t))$  for a.e. $t\in[t_0,T]$ and $\bar{u}(T)=U(T,\bar{x}(T))$. By Proposition \ref{ws-11} the condition \eqref{mwn} holds. Thus we can use Theorem \ref{Twon} [Invariance Theorem]. In view of this theorem we have $\bar{u}(t)\geq U(t,\bar{x}(t))$ for all $t\in[t_0,T]$. In particular
$$V(t_0,x_0)\,=\,\bar{u}(t_0)\,\geq\, U(t_0,\bar{x}(t_0))\,=\,U(t_0,x_0).$$
Therefore $V(t_0,x_0)\geq U(t_0,x_0)$ for all $(t_0,x_0)\in[0,T]\times\R^{\scriptscriptstyle N}$.

Next, we show that  $U\geq V$ on $[0,T]\times\R^{\scriptscriptstyle N}$. Obviously $U(T,x_0)=g(x_0)=V(T,x_0)$ for all\linebreak $x_0\in\R^{\scriptscriptstyle N}$. Fix $(t_0,x_0)\in\D U\cap [0,T)\times\R^{\scriptscriptstyle N}$. By Proposition \ref{flw1} the condition \eqref{swn} holds. So we can use Theorem \ref{Twi5} [Viability Theorem]. In view of this theorem there exists an absolutely continuous function  $(x,u):[t_0,T]\rightarrow\R^{\scriptscriptstyle N}\times\R$ with $(x,u)(t_0)=(x_0,U(t_0,x_0))$ which satisfies  $(\dot{x},\dot{u})(t)\in Q(t,x(t))$ for a.e. $t\in[t_0,T]$ and $u(t)\geq U(t,x(t))$ for all $t\in[t_0,T]$. Hence
$u(T)\geq U(T,x(T))=g(x(T))$ and
$$U(t_0,x_0)\,=\,u(t_0)\,=\,u(T)+\int_{t_0}^T\!\!-\dot{u}(t)\;dt\,\geq\, g(x(T))+\int_{t_0}^TL(t,x(t),\dot{x}(t))\;dt\,\geq\, V(t_0,x_0).$$
Therefore $U(t_0,x_0)\geq V(t_0,x_0)$ for all $(t_0,x_0)\in[0,T]\times\R^{\scriptscriptstyle N}$.
\end{proof}


\begin{thebibliography}{99}


\bibitem{IPA} \textsc{J.-P. Aubin}, {\em Viability Theory}, Birkh\"{a}user, Boston-Basel-Berlin  1991.

\bibitem{A-F} \textsc{J.-P. Aubin, H. Frankowska}, {\em Set-Valued Analysis}, Birkh\"{a}user, Boston-Basel-Berlin  1990.

\bibitem{B-CD} \textsc{M. Bardi, I. Capuzzo-Dolcetta}, {\em Optimal control and viscosity solutions of Hamilton-Jacobi-Bellman equations}, Birk\-h\"{a}user, Boston 1997.

\bibitem{GB} \textsc{G. Barles}, {\em Discontinuous viscosity solutions of first-order Hamilton-Jacobi equations: a guided visit}, Nonlinear Anal. 20 (1993), 1123--1134.

\bibitem{B-J} \textsc{E.N. Barron, R. Jensen}, {\em Semicontinuous viscosity solutions for Hamilton-Jacobi equations with convex Hamiltonians}, Comm.  Partial Differential Equations, {15} (1990), 1713--1742.

\bibitem{B-J-1991} \textsc{E.N. Barron, R. Jensen}, {\em Optimal control and semicontinuous viscosity solutions}, Proc. Am. Math. Soc., {113} (1991), 397--402.


\bibitem{B-B} \textsc{J. Bernis, P. Bettiol}, {\em Solutions to the Hamilton-Jacobi equation for Bolza problems with discontinuous time dependent data}, ESAIM Control Optim. Calc. Var., {26} (2020), 35 pages.

\bibitem{B-V} \textsc{P. Bettiol, R.B. Vinter}, {\em The Hamilton Jacobi equation for optimal control problems with discontinuous time dependence}, SIAM J. Control Optim., {55} (2017), 1199--1225.

\bibitem{C-S-2004} \textsc{P. Cannarsa, C. Sinestrari}, {\em Semiconcave functions, Hamilton-Jacobi equations, and optimal control}, Progress in Nonlinear Differential Equations
and their Applications, vol. 58, Birkhäuser, Boston,  2004.

\bibitem{LC} \textsc{L. Cesari}, {\em Optimization-theory and  applications, Problems with ordinary differential equations}, Springer, New York, 1983.

\bibitem{C-L-87} \textsc{M.G. Crandall, P.-L. Lions},  {\em Remarks on the existence and uniqueness of unbounded viscosity solutions of Hamilton-Jacobi equations}, Illinois J. Math., {31} (1987), 665--688.

\bibitem{DM-F-V} \textsc{G. Dal Maso, H. Frankowska}, {\em Value functions for Bolza problems with discontinuous Lagrangians and Hamilton-Jacobi inequalities}, ESAIM Control Optim. Calc. Var., {5} (2000), 369--393.

\bibitem{HF} \textsc{H. Frankowska}, {\em Lower semicontinuous solutions of Hamilton-Jacobi-Bellman equations}, SIAM J. Control Optim., {31} (1993), 257--272.

\bibitem{F-M-2013} \textsc{H. Frankowska, M. Mazzola}, {\em Discontinuous solutions of Hamilton-Jacobi-Bellman equation under state constraints}, Calc. Var. Partial Differential Equations, {46} (2013), 725--747.

\bibitem{F-P-Rz} \textsc{H. Frankowska, S. Plaskacz, T. Rzeżuchowski},\;\; {\em Measurable viability theorems and Hamilton-Jacobi-Bellman equation}, J. Differential Equations, {116} (1995), 265--305.


\bibitem{G} \textsc{G.H. Galbraith},  {\em Extended Hamilton-Jacobi characterization of value functions in optimal control}, SIAM J. Control Optim., {39} (2000), 281--305.


\bibitem{H-T} \textsc{J.W. Hagood, B.S. Thomson},  {\em Recovering a function from a Dini derivative}, Amer. Math. Monthly, {113} (2006), 34--46.

\bibitem{HI0} \textsc{H. Ishii}, {\em Uniqueness of unbounded solutions of Hamilton-Jacobi equations}, Indiana Univ. Math. J., vol. {33} (1984),  721--748.

\bibitem{HI} \textsc{H. Ishii}, {\em A generalization of a theorem of Barron and Jensen and a comparison theorem for lower semicontinuous viscosity solutions}, Proc. Roy. Soc. Edinburgh Sect. A, 131 (2001), 137--154.


\bibitem{AM2} \textsc{A. Misztela}, {\em The value function representing Hamilton-Jacobi equation with Hamiltonian depending on value of solution},  ESAIM Control Optim. Calc. Var., {20} (2014), 771--802.

\bibitem{AM0} \textsc{A. Misztela}, {\em On nonuniqueness of solutions of Hamilton-Jacobi-Bellman equations},  Appl. Math. Optim., {77} (2018), 599--611.

\bibitem{AM} \textsc{A. Misztela}, {\em Representation of Hamilton-Jacobi equation in optimal control theory with compact control set}, SIAM J. Control Optim., {57} (2019), 53--77.

\bibitem{AM1} \textsc{A. Misztela}, {\em Representation of Hamilton-Jacobi equation in optimal control theory with unbounded control set}, J. Optim. Theory Appl., {185} (2020), 361--383.

\bibitem{AM3} \textsc{A. Misztela, S. Plaskacz}, {\em An initial condition reconstruction in Hamilton–Jacobi equations}, Nonlinear Anal., {200} (2020), 15 pages.

\bibitem{AP} \textsc{A. Plaksin}, {\em Viscosity Solutions of Hamilton-Jacobi-Bellman-Isaacs Equations for Time-Delay Systems}, SIAM J. Control Optim., {59} (2021), 1951--1972.

\bibitem{P-Q} \textsc{S. Plaskacz, M. Quincampoix}, {\em On representation formulas for Hamilton Jacobi's equations\linebreak related to calculus of variations problems}, Topol. Methods Nonlinear Anal., {20} (2002), 85--118.

\bibitem{FR-S} \textsc{F. Rampazzo, C. Sartori},  {\em Hamilton-Jacobi-Bellman equations with fast gradient-dependence},\;\; Indiana Univ. Math. Journal, {49} (2000), 1043--1077.

\bibitem{RTR73} \textsc{R.T. Rockafellar}, {\em Optimal arcs and the minimum value function in problems of Lagrange}, Trans. Amer. Math. Soc., {180} (1973),  53--84.

\bibitem{RTR75} \textsc{R.T. Rockafellar}, {\em Existence theorems for general control problems of Bolza and Lagrange}, Adv. Math., {15} (1975),  312--333.

\bibitem{R-W} \textsc{R.T. Rockafellar, R.J.-B. Wets}, {\em Variational Analysis}, Springer-Verlag, Berlin 1998.


\end{thebibliography}
\end{document}